\documentclass[11pt]{amsart}
\usepackage{latexsym,mathrsfs,bm,float}
\usepackage[english]{babel}
\usepackage{paralist, ulem}
\usepackage{amsfonts,amssymb,amsmath}
\usepackage{amsthm}
\usepackage[latin1]{inputenc}
\usepackage{bm}
\usepackage{color} 
\usepackage{ifthen}

\usepackage{tikz}
\usetikzlibrary{calc,decorations.markings}

\newcommand{\eps}{\varepsilon}
\newcommand{\ee}{\varepsilon}
\newcommand{\vphi}{\varphi}
\newcommand{\R}{\mathbb{R}}
\newcommand{\Q}{\mathbb{Q}}
\newcommand{\C}{\mathbb{C}}

\newcommand{\N}{\mathbb{N}}

\newcommand{\Z}{\mathbb{Z}}

\newcommand{\E}{\mathcal{E}}
\newcommand{\Sm}{\Sigma}
\newcommand{\br}{{\boldsymbol r}}
\newcommand{\bL}{{\boldsymbol L}}
\newcommand{\bss}{{\boldsymbol s}}
\newcommand{\lambdab}{{\boldsymbol \lambda}}
\newcommand{\ang}[1]{\langle#1\rangle}

\newcommand{\J}{\mathcal{J}}

\newcommand{\lf}{{\mathfrak f}}

\newcommand{\tn}[1]{\textnormal{#1}}

\renewcommand{\mod}[1]{~\pr{\textnormal{mod}~#1}}
\newcommand{\Er}{{\mathcal E}}
\newcommand{\hessK}{{\mathfrak D}}
\newcommand{\Hess}{\operatorname{Hess}}

\newtheorem*{theo*}{Theorem}
\newtheorem{theo}{Theorem}
\newtheorem{theorem}[theo]{Theorem}
\newtheorem{ezer}{Exercise}

\newtheorem{conj}[ezer]{Conjecture}
\newtheorem{lemma}{Lemma}

\theoremstyle{definition}

\newtheorem{remark}{Remark}

\newtheorem*{rem*}{Remark}
\def\sumstar{\operatornamewithlimits{\sum\nolimits^*}}

\newcommand{\ssum}[1]{\sum_{\substack{#1}}}
\newcommand{\ssumstar}[1]{\operatornamewithlimits{\sum\nolimits^*}_{\substack{#1}}}

\newcommand{\pprod}[1]{\prod_{\substack{#1}}}

\newcommand{\pr}[1]{\mathchoice{\left( #1\right)}{(#1)}{(#1)}{(#1)}}
\newcommand{\floor}[1]{{\left\lfloor {#1} \right\rfloor}}

\newcommand{\e}[1][]{\ifthenelse{\equal{#1}{}}{{\rm e}}{\operatorname{e}\pr{ #1}}}

\newcommand{\Gal}{\operatorname*{Gal}}

\newcommand{\bu}{\boldsymbol{u}}

\newcommand{\df}{\mathrm{d}}

\newcommand{\comment}[1]{}
\newcommand{\Oh}{\mathcal{O}}
\DeclareMathOperator{\Li}{Li}
\DeclareMathOperator{\Lie}{{\mathfrak L}}
\DeclareMathOperator{\Vol}{Vol}
\DeclareMathOperator{\cs}{cs}
\DeclareMathOperator{\denom}{denom}
\newcommand{\bs}{\boldsymbol }
\DeclareMathOperator{\cz}{c_0}
\DeclareMathOperator{\deks}{s}

\renewcommand\Re{\operatorname{Re}}
\renewcommand\Im{\operatorname{Im}}
\newcommand\minus{\smallsetminus}
\newcommand{\1}{{\mathbf 1}}
\newcommand{\syf}[1]{\left(\!\left(#1\right)\!\right)}
\newcommand{\abs}[1]{\left|#1\right|}
\newcommand{\norm}[1]{\left\|#1\right\|}
\renewcommand{\bar}[1]{\overline{#1}}

\newcommand{\gh}{H}
\newcommand{\SL}{\tn{SL}}

\newcommand\numberthis{\stepcounter{equation}\tag{\theequation}}

\newcommand{\image}[2]{\includegraphics[width=#1\textwidth]{#2.png}}

\let\originalleft\left
  \let\originalright\right
\renewcommand{\left}{\mathopen{}\mathclose\bgroup\originalleft}
  \renewcommand{\right}{\aftergroup\egroup\originalright}

\numberwithin{equation}{section}

\title[Modularity and distribution of quantum knots invariants]{Modularity and value distribution of quantum invariants of hyperbolic knots}
\date{\today}

\author{S. Bettin}
\address{DIMA - Dipartimento di Matematica, Via Dodecaneso, 35, 16146 Genova, Italy}
\email{bettin@dima.unige.it}

\author{S. Drappeau}
\address{Aix Marseille Universit\'e, CNRS, Centrale Marseille, I2M UMR 7373, 13453 Marseille, France}
\email{sary-aurelien.drappeau@univ-amu.fr}

\begin{document}

\begin{abstract}
  We obtain an exact modularity relation for the $q$-Pochhammer symbol. Using this formula, we show that Zagier's modularity conjecture for a knot $K$ essentially reduces to the arithmeticity conjecture for $K$. In particular, we show that Zagier's conjecture holds for hyperbolic knots $K\neq 7_2$ with at most seven crossings.
  For $K=4_1$, we also prove a complementary reciprocity formula which allows us to prove a law of large numbers for the values of the colored Jones polynomials at roots of unity. We conjecture a similar formula holds for all knots and we show that this is the case if one assumes a suitable version of Zagier's conjecture.
\end{abstract}



\maketitle

\section{Introduction}


Among knot invariants, the colored Jones polynomials $\{J_{K,n}\}_{n\geq2}$ and the Kashaev invariants~$\{\ang{K}_N\}_{N\geq 2}$ are of particular interest, by their relation to quantum field theory, and the geometry of hyperbolic manifolds~\cite{Witten1994, MurakamiMurakami2001}. We refer to \textit{e.g.} \cite{MurakamiMurakami2001, Yokota2011} for their definitions; by~\cite{MurakamiMurakami2001}, the two invariants are related by~$\ang{K}_N=J_{K,N}(\e^{2\pi i /N})$. We refer to~\cite{Jones,Kashaev1995,Champanerkar_etal} for more results and references on this topic.

The Kashaev invariant is extended to a function on roots of unity by setting, for~$(h, k)=1$, 
$\J_{K,0}(\e^{2\pi i h/k}):=J_{K,k}(\e^{2\pi i h/k})$. For fixed~$k$, the values~$(J_{K,0}(\e^{2\pi i h/k}))_{(h, k)=1}$ are simply the Galois conjugates of $\ang{K}_N $ in~$\Q(\e^{2\pi i / k})$.
In the case of $K=4_1$, the simplest hyperbolic knot, we have explicitly
\begin{align}\label{jones41}
  \J_{4_1,0}(q)= \sum_{r=0}^{\infty} | (1-q)(1-q^2)\dotsb (1-q^r)|^2
\end{align}
for a root of unity $q$. In general, $\J_{K,0}(q)$ can be written as a series of this kind, involving a ratio of $q$-Pochhammer symbols of various indexes. See Section~\ref{sec:recip2-poch} for some more examples and the precise definition of~$\J_{K,0}(q)$ in the cases we will consider.

The \emph{volume conjecture}~\cite[Section~5]{MurakamiMurakami2001} predicts that 
\begin{equation}\label{eq:vc}
  \lim_{N\to\infty}\frac{\log |\ang{K}_N|}{N}=\lim_{N\to\infty}\frac{\log |\J_{K,0}(\e^{2\pi i /N})|}{N}=\frac{\norm{K}}{2\pi} v_3,
\end{equation}
where~$\norm{K}$ is related to the Gromov simplicial volumes of the complement of~$K$, and~$v_3$ is a suitable constant (the volume of the ideal regular tetrahedron in~${\mathbb H}^3$).
All the knots we will refer to in this paper will be hyperbolic; in this case, the asymptotic formula~\eqref{eq:vc} was conjectured by Kashaev~\cite{Kashaev1997}, and $\norm{K}v_3 = \tn{Vol} (K)$, the hyperbolic volume of the complement of $K$ in the $3$-sphere.
This is motivated by the analogy between the usual dilogarithm, which measures volumes of tetrahedra in the hyperbolic space, and the quantum dilogarithm, which are the building blocks of Kashaev's invariant.

This conjecture was extended in~\cite{Gukov2005} to a full asymptotic expansion, referred to as the \emph{arithmeticity conjecture} in~\cite{CalegariGaroufalidisZagier}, whereas the corresponding question for the imaginary part of the logarithm is conjectured to involve the Chern-Simons invariant $\cs (K)$ of $K$~\cite{MurakamiMurakamiEtAl2002, Gukov2005}. The arithmeticity conjecture has been proved for all knots with up to seven crossings in~\cite{AndersenHansen2006,Ohtsuki52,OhtsukiYokota6,Ohtsuki7}. We refer to~\cite{DimofteGukovEtAl2009,Murakami2017,MurakamiYokota2018,Murakami2011,KashaevTirkkonen}
and the references therein for more results and information on the volume conjecture.

In~\cite{Zagier2010}, Zagier studies several examples of what is called ``quantum modular forms''. Motivated by extensive numerical computations, he predicts that $\J_{K,0}$ satisfies an approximate modularity property which relates, in the limit as $x\to\infty$ among rationals of bounded denominator, $\J_{K,0}(\e^{2\pi i ({ax+b})/({cx+d})})$ with $\J_{K,0}(x)$ for any $(\begin{smallmatrix}a &b \\ c& d\end{smallmatrix})\in\SL(2,\Z)$. More specifically, given a hyperbolic knot $K$, the following conjecture is made (cf. also~\cite{Garoufalidis2018})

\begin{conj}[Zagier's modularity conjecture for $K$]\label{conj:zagier}
  For all~$\gamma\in\SL_2(\Z)$ such that\footnote{In what follows, a matrix in~$\SL(2, \R)$ acts on~$\C$ by homography.}~$\alpha:= \gamma(\infty) \in\Q$, there exist $C_K(\alpha)\in\C$ and a sequence~$(D_{K,n}(\alpha))_{n\geq 0}$ of complex numbers such that, for all~$M\in\N$ and~$x\in\Q$, with $x\to\infty$, there holds
  \begin{equation}\label{eq:J0-approxmod}
    \frac{\J_{K,0}(\e[\gamma(x))]}{\J_{K,0}(\e[x])} = \Big(\frac{2\pi}{\hbar}\Big)^{3/2}\e^{i\frac{\Vol (K)-i\cs (K)}{\hbar}} C_K(\alpha)\Big(\sum_{0\leq n < M} D_{K,n}(\alpha) \hbar^{n} + O(\hbar^{M})\Big), \quad \hbar:=\frac{2\pi i}{x-\gamma^{-1}(\infty)},
  \end{equation}
  where $\e[x]:=\e^{2\pi i x}$ and the implied constant depends at most on~$\alpha$, on the denominator of~$x$ and on $M$. Moreover, if~$F_K$ is the invariant trace field of $K$, and $F_{K,\alpha}:=F_K(\e[\alpha])$, then:
  \begin{itemize}
  \item $C_K(\alpha)$ is a product of rational powers of elements of~$F_{K,\alpha}$;
  \item $D_{K,n}(\alpha)\in F_{K,\alpha}$ for $n\geq0$.
  \end{itemize}
\end{conj}

In the case~$K = 4_1$, Garoufalidis and Zagier~\cite{GaroufalidisZagier} announced a proof of Theorem~\ref{th:1}, and also numerically investigated the conjecture for other knots.
The case of the $4_1$ knot is special and rather simpler than that of other knots, due to the fact that in this case all the summands in the definition~\eqref{jones41} of $\J_{4_1,0}$ are positive. One can then use Laplace's method to extract the asymptotic expansion~\eqref{eq:J0-approxmod}. In general, this positivity is not present and there is a remarkable amount of cancellation among the terms.
Indeed, $\J_{K,0}(\e[1/N])$ is typically exponentially smaller than the largest summands in its definition, and this prevents one from applying a direct estimation based on Laplace's method. We circumvent this serious obstacle by obtaining a new modularity relation, with a precise description of the holomorphic and periodic behaviour of the error terms, for the $q$-Pochhammer symbol. This symbol is of crucial importance in the theory of $q$-series and often appears in the theory of modular forms and combinatorics (see for example~\cite{Berndt2010} and references therein).
For $r\in\Z_{\geq0}$, it is defined as
\begin{equation}\label{eq:def-Pochhammer} 
(q)_r:=\prod_{i=0}^r(1-q)^r,\qquad q\in\C.
\end{equation}
When $|q|<1$, one can also take $r=+\infty$ and obtain the Dedekind $\eta$-function $\eta(z):=\e[z/24](\e[z])_{\infty},$ an important example of a (half-integral weight) modular form. As such, $\eta$ satisfies the relation 
$$\eta(\gamma z)=\vartheta(\gamma)(cz+d)^\frac12\eta(z),\qquad \gamma=(\begin{smallmatrix}a &b \\ c& d\end{smallmatrix})\in\SL(2,\Z),$$ 
for a certain ``multiplier system'' $\vartheta$ (see~\cite[Section 2.8]{Iwaniec1997}).
This modularity relation can be naturally extended to the partial product at root of unities. Indeed, in Theorem~\ref{th:Ir} below we show that for $\alpha \in\Q$, $1\leq r< \tn{den}(\gamma \alpha)$ we have
\begin{equation}\label{eq:mro}
\e[\tfrac {\gamma \alpha}{24}](\e[\gamma \alpha])_r= \vartheta(\gamma) \e[\tfrac{\alpha}{24}](\e[ \alpha])_{r'}\psi_\gamma(\alpha,r)
\end{equation}
for some $1\leq r'< \tn{den}( \alpha)$ and where $\psi_\gamma(\alpha,r)$ is an explicit function with suitable holomorphicity properties. We refer to Section~\ref{repsec} for the precise formulation of this reciprocity formula which we believe to be of independent interest.
 With this new tool, we can reduce Zagier's modularity conjecture to a slightly modified form of the arithmeticity conjecture, thus showing that the two conjectures are ``morally equivalent''.
In particular, we are able to prove the conjecture for all hyperbolic knots $K\neq 7_2$ with at most seven crossings, since for these the arithmeticity conjecture is known by works of Andersen and Hansen~\cite{AndersenHansen2006} (in the case~$K=4_1$), Ohtsuki~\cite{Ohtsuki52,Ohtsuki7} (in the case~$K=5_2$ and with~$7$ crossings) and Ohtsuki and Yokota~\cite{OhtsukiYokota6} (in the case of $6$ crossings).

\begin{theo}\label{th:1}
  Let~$K\neq 7_2$ be a hyperbolic knot with at most~$7$ crossings. Then Conjecture~\ref{conj:zagier} holds for~$K$. The constant~$C_K(\alpha)$ has the shape
  \begin{equation}
    C_K(\alpha) = \e\Big(\frac{\nu_K}2 \deks(\alpha)\Big) c^{\nu_K/2} \Lambda_{K, \alpha}^{1/c} \delta_K^{-1/2},\label{eq:expr-CK}
  \end{equation}
  where $c$ is the denominator of~$\alpha$, $\Lambda_{K, \alpha} \in F_{K,\alpha}$,~$\nu_K\in\Z$, and~$\delta_K \in F_K$, $\nu_K$ is given in Figure~\ref{fig:nuK} below and $\deks(\alpha)$ is the Dedekind sum (see~\eqref{eq:def-dedekind}).
\end{theo}

\begin{remark}
By the works~\cite{OhtsukiTakata2015,Ohtsuki52,OhtsukiYokota6,Ohtsuki7}, the number $\pm 2i \delta_K^{-1}$ can be interpreted as the conjugate of a twisted Reidemeister torsion of~$K$. Our method gives the constant term $C_K\cdot D_{K,0}$ as an explicit product of algebraic numbers; in Remark~\ref{c45} in Section~\ref{pmt1} we give as examples its value in the cases of $K=4_1$ and $K=5_1$.
\end{remark}
\begin{remark}
Recently, Calegari, Garoufalidis and Zagier~\cite{CalegariGaroufalidisZagier} made a more precise conjecture on $C_K(\alpha)$, predicting it naturally factors as $\mu_{K,8c}\cdot \eps_K(\alpha)^{1/c}/\sqrt{\delta_K}$, where $c$ is the denominator of $\alpha$, $\mu_{K,8c}$ is a $8c$ root of unity, $\eps_K(\alpha)$ is a unit of $F_{K,\alpha}$ and $\delta_K\in F_{K}$. We do not at present have such a precise description of~$C_K(\alpha)$. This would presumably require a fine understanding of the congruence sums~\eqref{defcs}.
\end{remark}
By the work of Ohtsuki~\cite{Ohtsuki7}, the arithmeticity conjecture is known also for $K=7_2$ and we expect that our method would give Theorem~\ref{th:1} also in this case. However, the proof of this case is more involved, so we decided to exclude this case for simplicity. In any case, we want to stress again that the scope of our work is more general and suggests that any proof of the arithmeticity conjecture should be adaptable into a proof of the modularity conjecture via the use of the reciprocity relation~\eqref{eq:mro}.

\medskip

The modularity and the volume conjectures likely don't give the full picture on the symmetries of $\J_{K,0}$ nor on its values at roots of unity.
Indeed, in the case of $K=4_1$ we can show that Theorem~\ref{th:1} can be complemented by a second reciprocity formula relating $X=\frac hk$ with $X'=\frac {\overline h}k$, where the overline indicates the multiplicative inverse modulo the denominator. This new reciprocity formula involves
the ``cotangent sum'' (which appears also in the main term in the variation, effective for other ranges of the parameters, of the reciprocity formula given in Theorem~\ref{th:4} below)
\begin{equation*}
  \cz\pr{ h/k}:=-\sum_{m=1}^{k-1}\frac mk\cot\pr{\frac{\pi  mh}{k}},\qquad (h,k)=1,k\geq1,
\end{equation*}
which is itself a quantum modular form~\cite{BettinConrey2013}
and has been widely studied due to its connection to the  B\'aez-Duarte-Nyman-Beurling criterion for the Riemann hypothesis (see, for example,~\cite{Bagchi2006,Baez-Duarte2003,Baez-DuarteBalazardEtAl2005,Vasyunin1995}).
\begin{theo}\label{th:2}
  Let $1\leq h\leq k$ with $(h,k)=1$. Then
  \begin{equation}\label{eq:frf}
   \frac{\J_{4_1,0}\big(\e[{\overline{h}}/k]\big)}{\J_{4_1,0}\big(\e[{\overline{k}/}h]\big) }=\exp\bigg( \frac{\Vol(4_1)}{2\pi} \frac kh+\E( h,k)\bigg),
  \end{equation}
  where  
  \begin{align}\label{eq:ber}
    \E( h,k) \ll 
    \frac1k \max_{0\leq r' <h}\Big|\ssum{1\leq n\leq r'}\cot\Big(\pi\frac{n\overline k}h\Big)\frac{ n}{h}\Big|+\frac {1}{h}\Big|\cz\Big(\frac{\overline k}h\Big)\Big|
    +\log \frac kh+\frac{k}{h^{2}}.
  \end{align}
\end{theo}

This may be compared with~\eqref{eq:J0-approxmod}. Note that~$\cs(4_1) = 0$, so there is no corresponding contribution on the right-hand side of~\eqref{eq:frf}. 

In this case as well, the reciprocity formula~\eqref{eq:frf} stems from a corresponding relation for the $q$-Pochhammer symbol (see Theorem~\ref{thp} below). Notice that, despite not giving a full asymptotic expansion,~\eqref{eq:frf} is completely uniform. In particular, it permits to be successfully iterated for ``typical'' roots of unity, 
allowing us to deduce the following law of large numbers for $\log \J_{4_1,0}$.

\begin{theo}\label{theo:cfep}
For $\alpha\in\Q\cap(0,1)$ with continued fraction expansion   $[0;b_1,\dots,b_r]$,  $b_r>1$, let 
$$ r(\alpha):=r,\qquad \Sm(\alpha):=\sum_{\ell=1}^rb_\ell.$$
Then, as~$\Sigma(\alpha)/r(\alpha)\to\infty$, we have
  \begin{equation}\label{eq:cfef}\begin{aligned}
      \log \J_{4_1,0}(\e[\alpha])&\sim \frac{\Vol(4_1)}{2\pi} \Sm(\alpha).
    \end{aligned}\end{equation}
In particular, for almost all roots of unity $q$ of order $n\leq N$, one has
  \begin{equation}\label{eq:asy}
    \log \J_{4_1,0}(q) \sim  \frac{12 }{\pi^2}\frac{\Vol(4_1)}{2\pi}\log n \log \log n
  \end{equation}
  as $N\to\infty$.

\end{theo}
Equation~\eqref{eq:asy} can be seen as a version of the volume conjecture~\eqref{eq:vc} for typical roots of unity. Indeed, the volume conjecture provides the asymptotic behavior of $\log \J_{4_1,0}$ at the root of unity $\e[1/N]$, whereas our result gives the asymptotic for almost all roots of unity of denominator $\leq N$. Notice then in both cases the leading constant involves the hyperbolic volume $\Vol(4_1)$, but the size of $\J_{4_1,0}$ changes dramatically.

Equation~\eqref{eq:cfef} is stronger than~\eqref{eq:asy}, which will be readily deduced by~\cite{BettinDrappeau}, and gives an asymptotic formula in most cases, e.g. when $\alpha$ is restricted to rational numbers with bounded $r(\alpha)$ as the denominator of $\alpha$ goes to infinity. In particular, it generalises the volume conjecture, which corresponds to the case $\alpha=1/N=[0;N]$. It is very likely that the assumption $\Sm(\alpha)/r(\alpha)\to\infty$ cannot be removed in general. Indeed, if for example $\alpha_n=F_{n-1}/F_{n}$ with $F_n$ the $n$-th Fibonacci number so that $\Sm(\alpha_n)=r(\alpha_n)+1=n-1$, then Theorem~\ref{eq:cfef} would give $\log \J_{4_1,0}(\e[\alpha_n])\sim C n$, with $C=\frac{\Vol(4_1)}{2\pi}\approx 0.323\dots$, whereas numerically it appears that $F(\alpha_n)$ grows like $C' n$, for $C'\approx 1.1$ (cf. also~\cite[Figure~6]{Zagier2010}).

\medskip 
Our proof of Theorems~\ref{theo:cfep} depends crucially on the positivity of the summands in~\eqref{jones41} which is missing if $K\neq 4_1$. Nonetheless, we expect a similar result holds for all hyperbolic knots. Also, since $\Sm(\alpha)$ is distributed according to a stable law~\cite{BettinDrappeau}, we expect the same to hold for $\log |\J_{K,0}(\e[\alpha])|$ for any hyperbolic knot $K$.

\begin{conj}\label{conj:zagier2}
  Let $K$ be a hyperbolic knot.
  There exists a constant $D_{K}\in\R$ such that for any interval $[a,b]\subset\R$ there holds
  \begin{equation}\label{eq:conj}
    \begin{split}
      &|Q_N|^{-1}\Big|\Big\{q\in Q_N\  \Big | \bigg(\frac{ \log |\J_{K,0}(q)|}{\frac{\tn{Vol}(K)}{2\pi} \log N}-\frac{12 }{\pi^2 } \log \log N -D_{K}\bigg)\, \in [a,b]\Big\}\Big|\\
      & \qquad = \int_a^b f_1(x; \tfrac6\pi, 1, 0) \df x + o(1)
    \end{split}
  \end{equation}
  as $N\to\infty$, where $Q_N$ is the set of roots of unity of order $\leq N$, and~$c\mapsto f_1(c; \tfrac6\pi, 1, 0)$ is the density of the stable law~$S_1(\tfrac6\pi,1,0)$. In particular,
  $$
  \log |\J_{K,0}(q)|\sim \frac{12 }{\pi^2 }\frac{\tn{Vol}(K)}{2\pi} \log n \log \log n
  $$
  for almost all roots of unity $q$ of order $n\leq N$ as $N\to\infty$.
\end{conj}

In~\cite{Zagier2010} Zagier discusses the continuity with respect to the real topology of $$\gh_{4_1}(h/k):=\log |\J_{4_1,0}(\e[h/k])|-\log |\J_{4_1,0}(\e[k/h])|$$ 
and suggests that $\gh_{4_1}$ is discontinuous but $\mathcal C^{\infty}$ from the right and the left at non-zero {rationals}\footnote{This continuity property at rationals follows from the modularity conjecture but only when approaching a rational $h/k=[0;b_1,\dots,b_r]$ with fractions essentially of the form $[0;b_1,\dots,b_r,N]$ with $N\to\infty$; and not, for example, with $[0;b_1,\dots,b_r,N_1,N_2]$ with both $N_1,N_2\in\N$ going to infinity.}
and continuous but not differentiable as one approaches irrational numbers. Using Lebesgue's integrability condition and~\cite{BettinDrappeau}, 
one can easily show that this continuity condition together with a suitable continuity condition at zero implies Conjecture~\ref{conj:zagier2}.

\begin{theo}\label{th:pz}
  Let $K$ be a hyperbolic knot. Assume the following:
  \begin{itemize}
  \item $\gh_{K}(h/k):=\log |\J_{K,0}(\e[h/k])|-\log |\J_{K,0}(\e[k/h])|$ has a limit as~$h/k$ tends to any positive irrational number, 
  \item $\gh_{K}(h/k)-\frac{\tn{Vol} (K)}{2\pi}\frac{k}h-\frac{3}2\log (k/h)$ is uniformly bounded.
  \end{itemize}
  Then Conjecture~\ref{conj:zagier2} holds with~$D_K = \frac{1-\gamma_0-\log 2}{\pi^2/12} + \frac{24}{\pi\tn{Vol}(K)}\int_0^1 \frac{\gh_{K}(1/t) - \frac{\tn{Vol}(K)}{2\pi t}}{1+t}\df t$, the function $\gh_{K}$ being extended to $\R_{>0}$ by taking limits over the rationals.
\end{theo}
\begin{remark}
  One could replace the second assumption in Theorem~\ref{th:pz} with the assumption that  $\gh_{K}(h/k)-\frac{\tn{Vol} (K)}{2\pi}\frac{k}h-\frac{3}2\log (k/h)$ is left and right continuous as $h/k$ approaches any rational number.
\end{remark}

In the case of torus knots, the invariant~$\J_{K,0}$ can still be constructed and a formula of type~\eqref{eq:J0-approxmod} is expected to hold with~$\tn{Vol}(K)$ replaced by~$0$.
In this situation, the works~\cite{BaladiVallee2005,BettinDrappeau} would suggest that $\frac{\log |\J_{K,0}(q)|}{\sqrt{\log n}}$ becomes distributed according to a Gaussian law. In this case however, the conditions of Theorem~\ref{th:pz} are not sufficient to conclude.

\medskip

In view of Theorem~\ref{th:2}, it is natural to wonder if also the function $\gh_{K}^*(h/k):=\log |\J_{K,0}(\e[\overline h/k])|-\log |\J_{4_1,0}(\e[\overline k/h])|$ could be regular at irrational points. We can answer this question in the negative in the case of $K=4_1$.

\begin{theo}\label{concor}
  For all $x\in[0,1]$ we have $\limsup_{y\to x^\pm,\, y\in\Q} |\gh^\ast_{4_1}(y)|=+\infty$.
\end{theo}

\begin{figure}[h]
  \centering
  \begin{tabular}[h]{cc}
    \image{0.4}{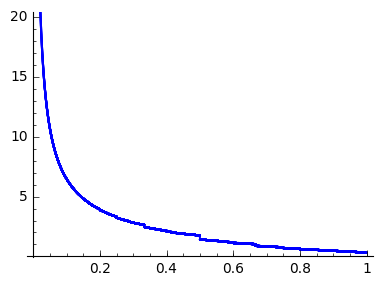} & \image{0.4}{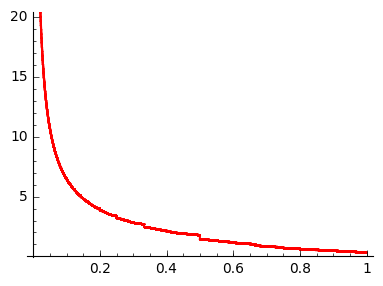}
  \end{tabular}
  \caption{Global graph of~$\gh_{4_1}$ and~$\gh^\ast_{4_1}$}
  \label{fig:graph-global}
\end{figure}

\begin{figure}[h]
  \centering
  \begin{tabular}[h]{cc}
    \image{0.4}{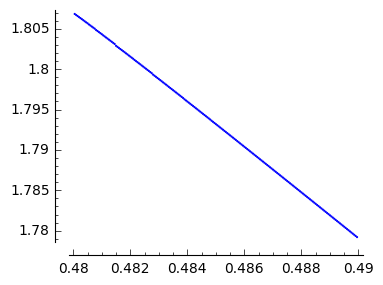} & \image{0.4}{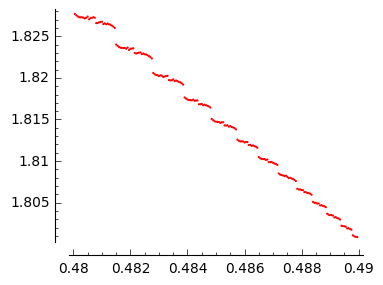}
  \end{tabular}  
  \caption{Graph of~$\gh_{4_1}$ and~$\gh^\ast_{4_1}$ around~$x=0.485$}
  \label{fig:graph-globalast}
\end{figure}

\subsection{Outline of the paper and sketch of the proofs}
Theorem~\ref{th:1} and Theorem~\ref{th:2} are both based on two new relations for the $q$-Pochhammer symbol, given in Theorems~\ref{th:Ir} (cf.~\eqref{eq:mro}) and~\ref{thp}.  These relations are proved in Section~\ref{repsec} and both make use of the Abel-Plana summation formula \cite[p.23]{Abel1992}, \cite[p.408]{Plana1820}, \cite[Chapter~8.3.1]{Olver1997}, which is a form of Euler-Maclaurin summation with an explicit form of the error term.
In the case of Theorem~\ref{th:Ir}, one starts by dividing the product in the definition~\eqref{eq:def-Pochhammer} of $(\e[\gamma \alpha])_{r}$ into appropriate intervals and congruence classes. One then take the logarithm and apply the summation formula to the resulting sum of the function~$\log(1-\e(z))$. As this function is close to a primitive of~$\pi\cot(\pi z)$, which has poles at integers, then through a residue computation 
one eventually arrives to the dual object~$(\e[ \alpha])_{r'}$. 
In the case of Theorem~\ref{thp}, the reciprocity relation for the $q$-Pochhammer symbol relates $\frac{\overline h}{k}$ and $\frac{\overline k}{h}$. In this case one starts by applying the simple relation $\frac{\overline h}{k}\equiv -\frac{\overline k}{h}+\frac1{hk}\mod 1$. After some initial manipulations, one is lead to consider sums of the function $\log(1-\cot(\pi \xi)\tan(\pi \xi z))+\cot(\pi \xi)\tan(\pi \xi z)$ for some choices of $\xi\in(0,1/2)$. This is again performed using Abel-Plana summation formula followed by a careful analysis, with particular care needed in the reassembling of various main terms.
 
In both our reciprocity relations for the $q$-Pochhammer symbol, we show that the error terms extend to holomorphic functions of controlled growth. This is crucial for the applications to the modularity relations for the Kashaev invariant for knots other than $4_1$. 

\medskip

Once the reciprocity formulas for the $q$-Pochhammer symbol are established the proofs of Theorem~\ref{th:1} and Theorem~\ref{th:2}, given in Section~\ref{pmt1} and~\ref{pmt2} respectively, follow in similar ways. We first split the sums in the definition of the Kashaev invariant into congruence classes (and suitable intervals) and apply the reciprocity relations, reducing the problem to that of estimating certain sums of exponentials of linear combinations of dilogarithms. These sums are very similar to the ones one needs to consider for the the volume conjecture with only two relevant differences: the variables of summation range over some convex space rather than some larger cubic regions, and inside the exponential we have also some new error terms. For all the knots we consider the first difference is easily treated since, as shown in Lemma~\ref{extra_terms}, the neglected terms are much smaller than the main terms (for other knots, such as $7_2$, a treatment as in~\cite[Section 8]{Ohtsuki7} should be possible). In the case of Theorem~\ref{th:1}, the second difference is also surpassed thanks the holomorphicity of the error terms mentioned above, since the complex analytic methods of~\cite{Ohtsuki52,OhtsukiYokota6,Ohtsuki7}, using Poisson summation and the saddle-point method, go through essentially unchanged. In the case of Theorem~\ref{th:2}, while we still have the holomorphicity of the error terms, the fact that the errors are not $o(1)$ forces us to use positivity to avoid possible cancellations in the main terms, thus restricting the applicability to the $4_1$ knot only.

\medskip

Theorems~\ref{theo:cfep}, \ref{th:pz} and~\ref{concor} are proved in Section~\ref{fise} and all use the reciprocity formulas~\eqref{eq:J0-approxmod}  and~\eqref{eq:frf} (the latter being more crucial) in conjunction with the recent work~\cite{BettinDrappeau} on the distribution of $\Sm(\alpha)$.
The difference between the reciprocity relations~\eqref{eq:J0-approxmod}  and~\eqref{eq:frf} can be better understood in terms of the continued fraction expansions $[0;b_1,\dots,b_m]$ of $h/k$ (for simplicity we assume $m$ odd). Indeed,~\eqref{eq:J0-approxmod} relates the values of $\J_{4_1,0}$ at $\e[{[0;b_1,\dots,b_m]}]$ and at $\e[{[0;b_\ell,\dots,b_m]}]$ provided that $b_\ell\to\infty$ for some $\ell\in\{1,\dots,m\}$ with all the other $b_i$ bounded, whereas~\eqref{eq:frf} relates the values of $\J_{4_1,0}$ at $\e[{[0;b_1,\dots,b_m]}]$ and at $\e[{[0;b_1,\dots,b_{m-1}]}]$ provided that $b_m\to\infty$ and that the other $b_i$ are not too large (for example $\log(b_i)=o(b_m)$ for all $i<m$ would suffice). 
Because of its uniformity,~\eqref{eq:frf} can be successfully iterated removing each time the last convergent $b_m$ from $\J_{4_1,0}(\e[{[0;b_1,\dots,b_m]}])$. We keep doing so untill we reach the last step for which we need to apply~\eqref{eq:J0-approxmod}. In this process, we pick up a main term of $\frac{\Vol(4_1)}{2\pi}b_m$ at each step and thus arrive to~\eqref{eq:cfef}. Equation~\eqref{eq:asy} then follows by the law of large numbers for $\Sm(\alpha)$ established in~\cite{BettinDrappeau}. 

Theorem~\ref{th:pz} follows a similar line, with the difference that in this case the previous argument and Conjecture~\ref{conj:zagier2} give that $\log \J_{4_1,0}(\e[\alpha])- \frac{\Vol(4_1)}{2\pi}\Sm(\alpha)$ can be well approximated by a differentiable function. The theorem then follows invoking again~\cite{BettinDrappeau}.

Finally, Theorem~\ref{concor} follows via a simple argument from Theorem~\ref{th:4}, a version of Theorem~\ref{th:2} which becomes useful when there is a middle partial quotient which  is extremely large.

\section*{Acknowledgment}
This paper was partially written during a visit of of S. Bettin at the Aix-Marseille University, a visit of S. Drappeau at the University of Genova, and a visit of both authors at ICTP Trieste. The authors thank these Institution for the hospitality and Aix-Marseille University, INdAM and ICTP for the financial support for these visits.

The authors wish to thank Don Zagier for useful discussions, Brian Conrey for putting us in contact with him, and Hitoshi Murakami for many helpful comments on this work.

S. Bettin is member of the INdAM group GNAMPA and his work is partially supported by PRIN 2017 ``Geometric, algebraic and analytic methods in arithmetic''.

\section*{Notation}
Given $D\subset \R$ and $f:D\to\R$, we write $\|f\|_{\infty, D}:=\sup_{t\in D}|f(t)|$. Also, given $t\in\R$, we write
$\|t\|:=\tn{dist}(t,\Z)$ and $\{t\}:=t-\lfloor t\rfloor$, where $\lfloor t\rfloor$ is the integer part of $t$. Given a property $P$, we define $\1_{P}$ (or $\1(P)$) to be $1$ if the property $P$ is satisfied and $0$ otherwise. 

All the implicit constants of the error terms are understood to be uniform in the various parameters unless otherwise indicated.

\section{Two reciprocity formulae for the $q$-Pochhammer symbol}\label{repsec}

\subsection{Abel-Plana's summation formula}

Our argument is based on the Abel-Plana summation formula.

We denote by $\gamma_\eps$ the following integration contour. 
\begin{figure}[H]\label{fig:contour}
  \centering
  \begin{tikzpicture}
    \draw (-3.2, 0) -- (3.2,0)
    (0, -1.5) -- (0, 1.5);

    \draw[blue, thick]
    (0,-0.7) arc (-90:0:0.7)  -- (3.2, 0)
    [postaction={decorate, decoration={markings, mark=at position 0.2 with {\arrow{latex}}}}]
    [postaction={decorate, decoration={markings, mark=at position 0.6 with {\arrow{latex}}}}];
    \draw[blue] (0.9,-0.4) node{$\gamma_\eps$};
    \draw[blue] (0.7,0.2) node{$\eps$};
    \draw[blue] (-0.35,-0.7) node{$-i\eps$};
  \end{tikzpicture}
\end{figure}

\begin{lemma}\label{lem:abelplana}
  Let~$\alpha, \beta,\beta'\in\R$ with~$\alpha \leq \beta<\beta'$. Let $f$ be an analytic function on a neighborhood of~$U:=\{z\in\C \mid \alpha\leq \Re(z)\leq \beta'\}\setminus\{\alpha,\beta\}$. Assume that the following holds~:
  \begin{enumerate}
  \item $f(z)$ is holomorphic at $\beta$ if $\beta$ is an integer, and otherwise $f(z)=o(|z-\beta|^{-1})$ as $z\to\beta$ with $z\in U$,
  \item $f(x\pm iy) = o(\e^{2\pi y}/y^2)$ as~$y\to+\infty$, uniformly in~$x\in[\alpha, \beta]$,
  \item $f$ is integrable on~$(\alpha, \beta)$.
  \end{enumerate}
  Then we have
  \begin{align}\label{eq:abelplana}
    \sum_{\alpha < n \leq \beta} f(n) = \int_{\alpha}^{\beta} f(t)\df t - C(f, \alpha) + C(f, \beta), 
  \end{align}
  where 
  $$
  C(f, \alpha):=\lim_{\eps\to0^+}\bigg(- i \int_{\gamma_\eps}  \frac{f(\alpha + it)\df{t}}{\e[-\alpha]\e^{2\pi t}-1}+ i \int_{\overline{\gamma_\eps}}  \frac{f(\alpha - it)\df{t}}{\e[\alpha]\e^{2\pi t}-1}\bigg). $$
\end{lemma}
\begin{proof}
The arguments in~\cite[Chapter~8, eq. (3.01)]{Olver1997} are readily adapted.
\end{proof}

For $k\in\Z_{\geq 0}$, $k\neq 1$, let $\tilde B_{k}(t)= {B}_{k}(\{t\})$ where $\{t\}$ is the fractional part of $t$ and ${B}_{k}$ is the $k$-th Bernoulli polynomial, and let $\tilde B_{1}(t)={B}_1(\{t\})$ for $t\notin\Z$ and $\tilde B_1(n)=0$ for $n\in\Z$.
We require the following computation. 
\begin{lemma}\label{lem:B-integral}
  For~$\ell\in\Z_{\geq0}$ and~$v\in[0, 1)$, we have
  $$ \int_{0}^\infty \Im\bigg( \frac{(-it)^\ell}{\e[v] \e^{2\pi t}-1}\bigg) \df t=  \frac{(-1)^\ell \tilde B_{\ell+1}(v)}{2(\ell+1)}. $$
  Moreover, for all~$\omega\in\C\minus(-\infty, 1]$, we have
  $$ \int_0^\infty \frac{\df t}{\omega  \e^{2\pi t}-1} = \frac{-1}{2\pi}\log(1-\omega^{-1})  $$
  with $\log$ being the principal determination.
\end{lemma}
\begin{proof}
  The second claim is easy to prove by expanding the fraction as a power series in~$\omega$, first for $\omega\in(1,\infty]$, and then by analytic continuation. To show the first claim, first we note that for~$\ell=0$, $v=0$, the fraction is a real number, and both sides evaluate to~$0$. We may therefore assume that~$\ell\geq 1$ or~$v\neq 0$. Then
  \begin{align*}
    \int_{0}^\infty \Im \bigg(\frac{(-it)^\ell}{\e[v] \e^{2\pi t}-1}\bigg)\df t = {}& \frac{1}{(2\pi)^{\ell+1}}\Im\bigg( \e[-v] (-i)^\ell \int_{0}^\infty \frac{t^\ell \e^{-t}\df t}{1-\e[-v] \e^{-t}}\bigg) \\
    = {}& \frac{\ell!}{(2\pi)^{\ell+1}}\Im{}\bigg( (-i)^\ell \sum_{n\geq 1}\frac{\e[-nv]}{n^{\ell+1}}\bigg)
  \end{align*}
  by~\cite[eq. (3.411.6)]{GZ}. We write~$\Im{}( (-i)^\ell\e[-nv] )= (-1)^{\ell+1} \cos(2\pi nv - \frac\pi2(\ell+1))$, and conclude by the Fourier expansion of Bernoulli polynomials~\cite[eq. (9.622.1)]{GZ}.
\end{proof}

\subsection{First reciprocity formula for the $q$-Pochhammer symbol}\label{ss:1rf}

We fix the notations as follows. Let~$N, d\geq 1$ be coprime. Let~$\alpha = p/q$ in reduced form, and~$\gamma = \left( \begin{smallmatrix} p & -\bar{q} \\ q & \bar{p} \end{smallmatrix} \right)$ for some~$(\bar{p}, \bar{q})$ satisfying~$p\bar{p} + q\bar{q} = 1$. Throughout the rest of Section~\ref{ss:1rf}, all error terms will be allowed to depend on~$d$ and $\gamma$ (but not on~$N$).
We write $h=Np-d\bar{q}$, $k=Nq+d\bar{p}$, so that
$$ x = \frac Nd, \qquad \gamma(x) = \frac{Np-d\bar{q}}{Nq+d\bar{p}} = \frac hk $$
and notice that this implies that $(h,k)=1$ (since $(h,k)$ divides $qh-pk=d$ and $(d,(h,k))=(p,q)=1$) and
\begin{align}
  \frac{h}{k}&=\frac{p}{q}-\frac{d}{kq},\label{redeq}\\
  \frac{d}{ kq}+\frac{k}{dq}&=\frac{N}{d}-\frac{h}k+\frac{p+\overline p}q.\label{expid}
\end{align}

We also recall from~\cite[p.45]{Iwaniec1997} that the Dedekind sum~$\deks(\alpha)$ is defined by
\begin{equation}
s(\alpha)= s(p, q) := \sum_{n=1}^{q-1} \frac{n}q \syf{\frac{pn}{q}}, \qquad \text{where }\syf{x} = \{x\} - \tfrac12,\label{eq:def-dedekind}
\end{equation}

For $z\in (\C \setminus \R)\cup(0,1)$, we let
\begin{equation}\label{eq:deff}
  \lf(z) := \log(1-\e[z])
\end{equation} 
taking the determination which is real on the positive imaginary axis. Notice that with this choice we have
\begin{align}
  &\lf(z)=\lf(1-z)+\pi i (2z-1); &&\label{ff1}\\
  &\lf(1+z)=\lf(z), \  \text{if }\Im(z)>0; && \lf(1+z)-\lf(z)=2\pi i, \  \text{if }\Im(z)<0;\label{ff2}\\
  &\lf(z)=\log(2\sin(\pi z))+ i\pi  (z-\tfrac12) \  \text{if }z\in(0,1).\label{ff3} 
\end{align}
Moreover, if $t\in(0,1)$, expanding the logarithm in its Taylor series
\begin{equation}\label{ff4}
  \sum_{g=1}^{q} \lf\Big(\frac{g-t}q\Big)=-\sum_{g=1}^{q}\sum_{m=1}^\infty\frac{\e[-tm/q]\e[mg/q]}{m}=\sum_{m=1}^\infty\frac{\e[-mt]}{m}=\lf(1-t)
\end{equation}
and the same formula holds for $t\in(\C \setminus \R)\cup(0,1)$ by analytic continuation.
Finally, for~$\lambda\in(\C \setminus \R)\cup(0,1)$, we let
\begin{equation}\label{dig}
  \Lie(\lambda) = \int_0^\lambda \lf(1-t)\df t - \frac{\pi i}{12}.
\end{equation}
The function~$\Lie$ is holomorphic in~$(\C \setminus \R)\cup(0,1)$. Note that~$\Li_2(1) = \pi^2/6$, so that whenever~$\Im(\lambda)<0$,
\begin{equation*}
\Lie(\lambda) = \int_0^\lambda \lf(1-t)\df t - \frac{\pi i}{12} = \frac1{2\pi i}\int_{\e[-\lambda]}^1 \log(1-z) \frac{\df z}z + \frac{\Li_2(1)}{2\pi i} =\frac{\Li_2(\e[-\lambda])}{2\pi i}.
\end{equation*}

Before stating the main theorem of this section, we define
\begin{equation}\label{defh}
  H_\kappa(u, v) := 
  \lim_{\eps\to0^+}\bigg( i \int_{\gamma_\eps}  \frac{\lf(u - it\kappa )\df{t}}{\e[v]\e^{2\pi t}-1}- i \int_{\overline{\gamma_\eps}}  \frac{\lf(u + it\kappa)\df{t}}{\e[-v]\e^{2\pi t}-1}\bigg)
\end{equation}
for $\kappa>0, u\in[0,1], v\in\R/\Z$. Notice that if~$v\not\in\Z$ we have
\begin{equation}\label{sph1}
  H_\kappa(u, v) = i
  \int_{0}^\infty   \bigg( \frac{\lf(u - it\kappa )}{\e[v]\e^{2\pi t}-1}-   \frac{\lf(u + it\kappa )}{\e[-v]\e^{2\pi t}-1}\bigg)\df{t},
\end{equation}
whereas if $v\in\Z$ and $0<u<1$ then
\begin{equation}\label{sph2}
  H_{\kappa}(u, 0) = -\frac12 \lf(u) +i\int_{0}^\infty    \frac{\lf(u - it\kappa)-\lf(u + it \kappa)}{\e^{2\pi t}-1}\df{t} ,
\end{equation}
as can seen by isolating the contribution of the two circular paths, which are both~$\pi i/2$ times the residue at~$0$.

\begin{theorem}\label{th:Ir}
  For~$1\leq r < k$, letting~$L := \lfloor rd/k\rfloor$, and~$\lambda = \{rd/k\}$, we have
  \begin{align*} 
    \frac{(\e[\gamma x])_r\e[\frac{\gamma x}{24}]}{(\e[x])_L\e[\frac{x}{24} ]} =\exp\bigg(&
    \frac{\pi i(p+\overline p)}{12q}-\pi i \deks(p,q)-\frac{\pi i }4+\frac12 \log\frac kd+
    \frac{k}{qd} \Lie(\lambda)+\Er_r(\lambda,d/k)\bigg), 
  \end{align*}
  where $\deks$ is the Dedekind sum~\cite[p.45]{Iwaniec1997} and for $\lambda\in[0,1)$, $\kappa>0$ and, for $s\in\Z/q\Z$,
  \begin{align} 
    \Er_s(\lambda,\kappa)&:=- H_{\kappa}\Big(\frac{\ang{ps} -\lambda}{q},0\Big)-\hspace{-1em} \ssum{g=1\\ g\not\equiv ps\mod q}^q i\int_{0}^\infty   \bigg( \frac{\lf(\frac{g-\lambda}q - it\kappa)}{\e[ \frac{g\bar{p}-s}q]\e^{2\pi t}-1}-   \frac{\lf(\frac{g-\lambda}q + it\kappa)}{\e[- \frac{g\bar{p}-s}q]\e^{2\pi t}-1}\bigg)\df{t}, 
    \label{defg}
  \end{align}
  with $\ang{n}$ indicating the representative of $n\mod q$ in~$[1,q]$. Moreover, for all~$s\mod{q}$ and~$\kappa\in(0, 1]$, the function~$\lambda\mapsto \Er_s(\lambda,\kappa)$ is defined and holomorphic in the strip~$\{\Re(\lambda)\in[0, 1)\}$, and
  \begin{align} 
    \Er_s(\lambda,\kappa)&\ll \abs{\log\abs{1-\lambda}} + \log(1/\kappa)
    \label{bdfg}
  \end{align}
  uniformly for $\Re(\lambda)\in[0,1)$ and~$\Im(\lambda)\ll 1$.
\end{theorem}
\begin{remark}
  If $\lambda\in(0,1)$ or if $s\not\equiv 0\mod q$ by~\eqref{sph2} we can write $\Er_\lambda$ as 
  \begin{align} 
    \Er_s(\lambda,\kappa)&:= \frac12 f\Big(\frac{\ang{ps}-\lambda}q\Big)- \sum_{g=1}^q i\int_{0}^\infty   \bigg( \frac{\lf(\frac{g-\lambda}q - it\kappa)}{\e[ \frac{g\bar{p}-s}q]\e^{2\pi t}-1}-   \frac{\lf(\frac{g-\lambda}q + it\kappa)}{\e[- \frac{g\bar{p}-s}q]\e^{2\pi t}-1}\bigg)\df{t}.
    \label{defg2}
  \end{align}
  We remark that if $(q,d)=1$ (i.e. $(k,d)=1$), one cannot have $\{rd/k\}=0$. \end{remark}
\begin{remark}\label{remplr}
  Notice that by~\eqref{expid} we have ${\e[\frac{\gamma x}{24}]}/{\e[\frac{x}{24} ]}\cdot \exp( - \frac{\pi i(p+\overline p)}{12q} )=\exp(-\frac{\pi i}{12} (\frac{d}{ kq}+\frac{k}{dq}))$.
\end{remark}
In order to prove Theorem~\ref{th:Ir}, we require some properties of the function
$H_\kappa(u, v)$.
\begin{lemma}\label{lemmah}
  We have~:
  \begin{enumerate}
  \item\label{prop-B-sym} ${H_\kappa(1-u, -v)} + H_\kappa(u, v)= - 2{\pi i} B_1(u)\tilde B_1(v) + \pi i\kappa \tilde B_2(v)$ for~$v\notin\Z$,
  \item\label{prop-B-f} $H_\kappa(1, v) - H_\kappa(0, v) = \lf(\{-v\}) $ for~$v\notin\Z$,
  \item\label{prop-B-pole} $H_\kappa(1, 0) =   - \frac{\log(  \kappa) }{2 }-   \frac{\pi  i   }{4  } + \frac{ \pi i \kappa}{12} $.
  \end{enumerate}
\end{lemma}
\begin{proof}
  Let $v\notin\Z$. By~\eqref{ff1}, Lemma~\ref{lem:B-integral} and~\eqref{ff3} we have
  \begin{align*}
    {H_\kappa(1-u, -v)} + H_\kappa(u, v)&
    =  i
    \int_{0}^\infty   \bigg(\frac{\lf(1-u - it\kappa )}{\e[-v]\e^{2\pi t}-1}- \frac{\lf(1-u + it\kappa )}{\e[v]\e^{2\pi t}-1}\bigg)\df{t}+{}\\
    &\quad +i
    \int_{0}^\infty   \bigg(    \frac{\lf(u - it\kappa )}{\e[v]\e^{2\pi t}-1}-\frac{\lf(u + it\kappa )}{\e[-v]\e^{2\pi t}-1}\bigg)\df{t}\\
    &= -\int_{0}^\infty   \bigg(\frac{(1- 2u)\pi   -  2\pi i t\kappa}{\e[-v]\e^{2\pi t}-1}- \frac{
      (1- 2u)\pi   + 2\pi i t\kappa
    }{\e[v]\e^{2\pi t}-1}\bigg)\df t\\
    &=2\pi i(1- 2u) \Im \int_{0}^\infty    \frac{ \df t}{\e[v]\e^{2\pi t}-1}-4\pi i\kappa \Im \int_{0}^\infty   \frac{-i t\df t}{\e[v]\e^{2\pi t}-1}\\
    &= -2{\pi i} B_1(u)\tilde B_1(v) +\pi i\kappa \tilde B_2(v).
  \end{align*}
  Also, by~\eqref{ff2}, we have
  \begin{align*}
    H_\kappa(1, v)-  H_\kappa(0, v) &= i
    \int_{0}^\infty   \bigg( \frac{\lf(1 - it\kappa )}{\e[v]\e^{2\pi t}-1}-   \frac{\lf(1 + it\kappa )}{\e[-v]\e^{2\pi t}-1}\bigg)\df{t}+{}\\
    &\quad + i \int_{0}^\infty   \bigg( \frac{\lf(- it\kappa )}{\e[v]\e^{2\pi t}-1}-   \frac{\lf( it\kappa )}{\e[-v]\e^{2\pi t}-1}\bigg)\df{t}\\
    &=-\int_{0}^\infty    \frac{2\pi }{\e[v]\e^{2\pi t}-1}\df{t}=\lf(\{-v\}).
  \end{align*}
  Finally, by~\eqref{ff1} and~\eqref{ff2} as $\eps\to0^+$ we have
  \begin{align*}
    & i \int_{\gamma_\eps}  \frac{\lf(1 - it\kappa )\df{t}}{\e^{2\pi t}-1}- i \int_{\overline{\gamma_\eps}}  \frac{\lf(1 + it\kappa)\df{t}}{\e^{2\pi t}-1}\\
    &\qquad= i \int_{\gamma_\eps}  \frac{\lf( it\kappa ) \df{t}}{\e^{2\pi t}-1}- i \int_{\overline{\gamma_\eps}}  \frac{\lf( it\kappa)\df{t}}{\e^{2\pi t}-1}- \int_{{\gamma_\eps}}  \frac{\pi  \df{t}}{\e^{2\pi t}-1} + \int_{0}^\infty \frac{  2\pi it\kappa\df{t}}{\e^{2\pi t}-1}+o(1).
  \end{align*}
  The last two integrals can be easily computed and contribute $-\frac{i\pi}4+\frac{\log(2\pi \eps)}2$ and $\frac{\pi i \kappa}{12}$ respectively. The contributions of the interval $(\eps,\infty)$ in the first two integral cancel out. Thus, since $\lf( it\kappa)=\log(2\pi  t \kappa)+o(1)$ as $|t|\to 0$ with $-\pi/2<\arg t<\pi/2$, we have that the first two integral contribute
  \[
    i \int_{C_\eps}  \frac{\log(2\pi  t \kappa) \df{t}}{\e^{2\pi t}-1}=-\frac{\log(2\pi \eps\kappa)}{2}+o(1)
  \]
  where $C_\eps$ is the semicircle centered at the origin going from $-i\eps$ to $i\eps$ counter-clockwise. We then have 
  \[
    H_\kappa(1,0)=-\frac{\log \kappa}{2}-\frac{i\pi}4+\frac{\pi i \kappa}{12}.
  \]
\end{proof}
\begin{lemma}\label{lwb}
  Let $0<\kappa\leq1 , \Re(u)\in [0,1], v\in\R/\Z$ with $u\neq0,1$, and~$A\geq 1$. 
  Then
  $$ H_\kappa(u,v) = -\log\abs{u} + O_{v, A}(\log(2/\kappa)) $$
  uniformly in $u,\kappa$ with~$\Im(u)\leq A$.
\end{lemma}
\begin{proof}
  The case $v\notin\Z$, $\Re(u)\in [0,1]$ 
   is an easy consequence of the bound 
  \begin{align}\label{tb}
    \lf(x+it)\ll |t| +|\log |t||,\quad  (x\in[0, 1], t\in\R_{\neq 0}).
  \end{align}
  Now assume $v\in\Z$ and~$\Re(u)\in(0, 1/2]$. We recall~\eqref{sph2}. By~\eqref{tb} the contribution to the integral from the interval $[2A,\infty)$ is $O(\log(2/\kappa))$. Next, we write
  $$ I(u, \kappa) := i\int_{0}^{2A} \frac{\lf(u - it\kappa)-\lf(u + it \kappa)}{\e^{2\pi t}-1}\df{t} = -\pi i\int_0^{2A} \int_{u-it\kappa}^{u+it\kappa} (\cot(\pi z)+i)\df z \frac{\df t}{\e^{2\pi t}-1}. $$
  Note that~$\cot(\pi z)  = 1/(\pi z) + O(1)$ uniformly for~$\Re(z)\in[0, 1/2]$, and that~$\frac1{\e^{2\pi t}-1} = \frac1{2\pi t} + O(1)$ for~$t>0$. Therefore,
  $$ I(u, \kappa) = \frac{1}{2\pi i}\int_0^{2A} \int_{u-it\kappa}^{u+it\kappa} \frac{\df z}z \frac{\df t}{t} + O_A(\log(2/\kappa)). $$
  Changing variables~$z \gets \abs{u} z$ and~$t\gets t\abs{u}/\kappa$, we get
  $$ I(u, \kappa) = \frac{1}{2\pi i}\int_0^{2A \kappa / \abs{u}} \int_{u'-it}^{u'+it} \frac{\df z}z \frac{\df t}{t} + O_A(\log(2/\kappa)), $$
  where~$\abs{u'}=1$. For~$t\leq 1/2$, we may bound the~$z$-integral by~$O(t)$, while
  for~$t\geq 2$, we have~$\int_{u'-it}^{u'+it} \frac{\df z}z = \pi i + O(1/t)$. We deduce
  $$ I(u, \kappa) = \frac{1}{2\pi i}\int_{1/2}^{2} \int_{u'-it}^{u'+it} \frac{\df z}z \frac{\df t}{t} + O_A(\log(2/\kappa)) - \frac12\log\abs{u}. $$
  The double integral here is bounded independently of~$u'$, and so finally
  $$ I(u, \kappa) = O_A(\log(2/\kappa)) - \frac12\log\abs{u}. $$
  We deduce~$H_\kappa(u, 0) = -\log\abs{u} + O_A(\log(2/\kappa))$ for~$0\leq\Re(u)\leq 1/2,$ $u\neq0$. On the other hand, by computations similar to Lemma~\ref{lemmah}, we get
  $$ H_\kappa(1-u, 0) + H_\kappa(u, 0) = -\lf(u) - \frac{\pi i}2 (2u-1) - \frac{\pi i \kappa}{6}, $$
  from which we get the claimed behaviour for all~$u$.
\end{proof}
\begin{proof}[Proof of Theorem~\ref{th:Ir}]
  For $0\leq\ell\leq L$, let~$r_\ell  = \ell k/d$. We split $(\e[\gamma x])_r$ as 
  $
  (\e[\gamma x])_r=\prod_{\ell =0}^LP_\ell^L
  $, where for~$0\leq \ell  < L$,  
  $$ P_\ell ^L = \prod_{r_\ell  < n \leq r_{\ell +1}} \bigg(1-\e[\frac{nh}k]\bigg) \qquad \text{and} \qquad P_L^L = \prod_{r_\ell  < n \leq r} \bigg(1-\e[\frac{nh}k]\bigg). 
  $$
  First we focus on the case $0\leq \ell<L$. By~\eqref{redeq} and by periodicity we have
  \begin{align*}
    P_\ell ^L ={}&  \prod_{r_\ell  < n \leq r_{\ell +1} } \bigg(1-\e[\frac{np}q - \frac{nd}{kq}]\bigg) =\prod_{a=1}^{q}\prod_{\frac{r_\ell -a}q < m \leq \frac{r_{\ell +1}-a}q} \bigg(1-\e[\frac{ap}q - \frac{d(a+mq)}{kq}]\bigg) \\
    ={}& \prod_{a=1}^{q}\prod_{\frac{r_\ell -a}q < m \leq \frac{r_{\ell +1}-a}q} \bigg(1-\e[\frac{g_a}q - \frac{d(a+mq)}{kq}]\bigg),
  \end{align*}
  where $g_a$ is the representative of the class $pa\mod q$ contained in $[\ell+1,\ell+q]$, so that in the last line each $\e[\,]$ is computed at a number in $(0,1)$. It follows that we can write
  \begin{align*}
    P_\ell ^L ={}&\exp\bigg(\sum_{a=1}^{q}\sum_{\frac{r_\ell -a}q < m \leq \frac{r_{\ell +1}-a}q} f\bigg(\frac{g_a}q - \frac{d(a+mq)}{kq}\bigg) \bigg)
  \end{align*}
  and that $0<\frac{g_a}q - \frac{dr_{\ell+1}}{kq}<1$ whenever $\frac{r_{\ell +1}-a}q$ is an integer.
  We can then apply Abel-Plana formula in the form of Lemma~\ref{lem:abelplana}.   
  Note that
  \begin{align*}
    \ssum{a=1}^{q} \int_{\frac{r_\ell -a}q}^{\frac{r_{\ell +1}-a}q} f\Big(\frac{g_a}q- \frac{d(a+tq)}{kq}\Big)\df t &= \frac k{qd} \ssum{a=1}^{q} \int_\ell ^{\ell +1} f\Big(\frac{g_a}q-\frac{t}q\Big)\df t \\
    &=\frac k{qd} \ssum{g=1}^{q} \int_0 ^{1} f\Big(\frac{g-t}q\Big)\df t=  \frac k{qd} \int_0 ^{1} \lf(1-t)\df t = 0
  \end{align*}
  by~\eqref{ff4}.  Therefore, by Lemma~\ref{lem:abelplana}, equation~\eqref{defh} and the definition of $r_\ell$,
  $$ P_\ell ^L = \exp\bigg(\ssum{a=1}^{q} H_{d/k}\Big(\frac{g_a}q-\frac\ell q, \frac{a}q - \frac{\ell k}{qd}\Big) - \ssum{a=1}^{q}  H_{d/k}\Big(\frac{g_a}q-\frac{\ell+1}q, \frac{a}q - \frac{(\ell +1)k}{qd}\Big)\bigg). $$
  Now, $\frac{k}{qd}=\frac{N}{d}+\frac{\overline p}q$ and $1\leq g_a-\ell\leq q$ with $pa\equiv g_a\mod q$, thus by a change of variable
  and multiplying this equality over~$0\leq \ell <L$, we get
  \begin{align*}
    \prod_{\ell =0}^{L-1} P_\ell ^L ={}&\exp\bigg( \sum_{\ell =0}^{L-1}\bigg(\sum_{g=1}^{q} H_{d/k}\Big(\frac {g}q, \frac{g\overline p}q - \frac{\ell N}d\Big) - \sum_{g=0}^{q-1} H_{d/k}\Big(\frac {g}q, \frac{g\overline p}q - \frac{(\ell +1)N}d\Big) \bigg)\bigg) \\
    ={}& \exp\bigg(\sum_{\ell =0}^{L-1} \Big(H_{d/k}\Big(1, -\frac{\ell N}d\Big) - H_{d/k}\Big(0, -\frac{(\ell +1)N}d\Big)\Big) \numberthis\label{eq:sumLinit}\\
    & + \sum_{g=1}^{q-1} \Big( H_{d/k}\Big(\frac gq, \frac{g\bar{p}}q\Big) - H_{d/k}\Big(\frac gq, \frac{g\bar{p}}q - \frac{LN}d\Big) \Big)\bigg). 
  \end{align*}
  We treat~$P_L^L$ in the same way. First,
  $$
  P_L^L 
  ={}\exp\bigg( \sum_{a=1}^{q}\ssum{\frac{r_L -a}q < m \leq \frac{r-a}q} f\Big(\frac{g_a}q - \frac{d(a+mq)}{kq}\Big)\bigg). 
  $$
  Note that by~\eqref{ff4} and~\eqref{dig}
  \begin{align*}
    \sum_{a=1}^{q} \int_{(r_L -a)/q}^{(r-a)/q} f\Big(\frac{g_a}q - \frac{d(a+qt)}{kq}\Big)\df t  &= \frac k{qd} \ssum{a=1}^{q} \int_L ^{L+\lambda} f\Big(\frac{g_a}q-\frac{t}q\Big)\df t \\
    &=\frac k{qd}  \int_{0}^\lambda \lf(1-t)\df t= \frac k{qd}  \Lie(\lambda)+\frac{\pi i k }{12 qd}\\
  \end{align*}
  where we recall that~$\lambda = \{rd/k\}$. Applying Lemma~\ref{lem:abelplana}, we therefore find
  \begin{equation}\label{eq:sumLlast}
    P_L^L = \exp\bigg(\frac k{qd}  \Lie(\lambda)+\frac{\pi i k }{12 qd} + \sum_{g=1}^{q} \Big(H_{d/k}\Big(\frac gq, \frac{g\bar{p}}q-\frac{LN}d\Big) - H_{d/k}\Big(\frac{g-\lambda}q, \frac{g\bar{p}-r}q\Big)\Big)\bigg).
  \end{equation}
  Multiplying the equalities~\eqref{eq:sumLinit} and~\eqref{eq:sumLlast} and recalling that $(\e[\gamma x])_r = \prod_{\ell =0}^L P_\ell^L$, we obtain
  \begin{align*}
    (\e[\gamma x])_r
    ={}& \exp\bigg(\frac k{qd}  \Lie(\lambda)+\frac{\pi i k }{12 qd} + \sum_{g=1}^{q-1} H_{d/k}\Big(\frac gq, \frac{g\bar{p}}q\Big) - \sum_{g=1}^q H_{d/k}\Big(\frac{g-\lambda}q, \frac{g\bar{p}-r}q\Big) +{} \\
    {}& \hspace{3em} + H_{d/k}(1, 0)+ \sum_{1\leq \ell  \leq L} \Big(H_{d/k}(1, -\tfrac{\ell N}d) - H_{d/k}(0, -\tfrac{\ell N}d)\Big)\bigg) .
  \end{align*}
  By Lemma~\ref{lemmah} (1) we have
  \begin{align*}
    \sum_{g=1}^{q-1} H_{d/k}\Big(\frac gq, \frac{g\bar{p}}q\Big)&=-\pi i \sum_{g=1}^{q-1} \tilde B_1\Big(\frac gq\Big)\tilde B_1\Big( \frac{g\bar{p}}q\Big)+\frac{\pi i d}{2 k} \sum_{g=1}^{q-1}\tilde B_2\Big(\frac gq\Big)\\
    &=-\pi i \deks(p,q)+\frac{\pi i d}{12 kq}-\frac{\pi i d}{12 k}
  \end{align*}
  since $\sum_{g=0}^{q-1}B_2(g/q)=\frac{1}{6q}$ and $B_2(0)=1/6$. Also, by Lemma~\ref{lemmah} (3) and (2) we have
  $H_{d/k}(1, 0)=  \frac{\log(  k/d) }{2 }-   \frac{\pi  i   }{4  } + \frac{ \pi i d}{12k}$ and
  \[
    \sum_{1\leq \ell  \leq L} \Big(H_{d/k}(1, -\tfrac{\ell N}d) - H_{d/k}(0, -\tfrac{\ell N}d)\Big)= \sum_{1\leq \ell  \leq L}\lf(\{{\ell N}/{d}\}).
  \]
  Thus, 
  \begin{align*}   (\e[\gamma x])_r={} (\e[N/d])_L\exp\bigg(&\frac k{qd}  \Lie(\lambda)+\frac12 \log\frac kd  - \sum_{g=1}^q H_{d/k}\Big(\frac{g-\lambda}q, \frac{g\bar{p}-r}q\Big)+{}\\
    &-\pi i \deks(p,q)-\frac{\pi i }4+\frac{\pi i d}{12 kq}+\frac{\pi i k }{12 qd} \bigg).
  \end{align*}
  By~\eqref{expid} and~\eqref{sph1} we then obtain the claimed result. The bound~\eqref{bdfg} follows by Lemma~\ref{lwb}.
\end{proof}

\subsection{Sums of cotangents}

In this section and the next, we introduce the following extension of the Landau~$O$-symbol. Let~$D_1\subset D_2\subset \C$ be two sets given by the context, and~$g: D_2 \to \R_+$. We will write
\begin{equation}
  f(z) = \Oh_{(p_1, p_2, \dotsc)}^z(g(z)) \qquad (z\in D_1)\label{eq:def-Oh}
\end{equation}
whenever, for any fixed choice of the parameters~$p_1, p_2, \dotsc$, there exists a function~$\varphi$ holomorphic on~$D_2$, which satisfies~$\vphi(z) = f(z)$ for~$z\in D_1$, and~$\abs{\vphi(z)}\ll g(z)$ for~$z\in D_2$, the implied constant being \textit{uniform} 
in the parameters~$p_1, p_2, \dotsc$.
The additional information is the holomorphic behaviour in~$z$, which may become useful when studying knots other than~$4_1$. However we stress that, in the present work, a later obstacle (possible cancellation of main terms) imposes the restriction~$K = 4_1$, and in this case, we do not require holomorphicity of error terms.

\begin{lemma}\label{tan_ineq}
  For $0<\alpha<\frac12$, we have
  \begin{align}
    \cot(\pi\alpha) \tan(\pi\alpha x) & {} \leq x & \hspace{8em} (0<x\leq 1), \label{eq:tanineq-1} \\
    x(1-(2\alpha)^2)< \cot(\pi\alpha) \tan(\pi\alpha x) &{} \leq O_\ee(x) &\hfill (0<x\leq \tfrac1\alpha(\tfrac12-\eps)). \label{eq:tanineq-2}
  \end{align}
\end{lemma}
\begin{proof}
  Assume~$0<x \leq \frac1\alpha(\frac12-\eps)$. The function~$x\mapsto \frac{\tan(x\lambda)}{x}$ is increasing on~$(0, \frac{\pi}{2\lambda})$ for all~$\lambda>0$. It follows that $\pi\alpha\leq \tan(\pi\alpha x)/x \leq O_\ee(\alpha)$, while~$\tan(\pi\alpha x)/x \leq \tan(\pi\alpha)$ if~$x\leq 1$. Thus, it suffices to show the bounds
  $$ 1-(2\alpha)^2 < \pi \alpha \cot(\pi \alpha), \qquad \tan(\pi\alpha ) \cot(\pi \alpha)\leq 1, \qquad \alpha \cot(\pi \alpha) \ll 1. $$
  The second is trivial, while the third follows from elementary properties of~$\cot$. For the first, we note that by the Taylor expansion of~$\cot$~\cite[1.411.7]{GZ}, we have~$\phi(\alpha) := \frac{1-\pi\alpha \cot(\pi \alpha)}{\alpha^2} = \sum_{j\geq 0} c_j \alpha^{2j}$ for~$\alpha\in(0, 1)$, where~$c_j>0$. In particular $\phi$ is increasing, and we conclude by~$\phi(\frac12)=4$.
\end{proof}
\begin{lemma}\label{tan_ineq2}
  Let $\eps>0$. For $0<\alpha\leq \frac12-\eps$ and $|\Re(\alpha z)|\leq \frac12-\eps$ we have
  \begin{align*}
    1-\cot(\pi\alpha)\tan(\pi\alpha z)=(1-z)\big(1+O_\eps(\alpha^2(|z|^2+|z|))\big).
  \end{align*}
\end{lemma}
\begin{proof}
  If $|z|\leq 1/2$, expanding in Taylor series we obtain
  \begin{align*}
    \cot(\pi\alpha)\tan(\pi\alpha z)=    \cot(\pi\alpha)\pi\alpha z(1+O(\alpha^2|z|^2))= z(1+O(\alpha^2))(1+O(\alpha^2|z|^2))
  \end{align*}
  and the claimed result follows. If $|z|\geq1/2$, we observe that the left hand side is equal to $\cot(\pi\alpha)(\tan(\pi\alpha )-\tan(\pi\alpha z))$ and so, since $\cot(\pi\alpha)\pi\alpha=1+O_\eps(\alpha^2)$, we are required to show that
  \begin{equation}
  \frac{\tan(\pi \alpha z)-\tan(\pi \alpha)}{\pi \alpha(z-1)}
  = \frac1{\alpha z-\alpha}\int_\alpha^{\alpha z} \frac{\df v}{\cos(\pi v)^2} =  1 + O_\ee( \abs{\alpha z}^2).\label{eq:lemmatan-integ}
\end{equation}
  This follows for $\abs{\Im(\alpha z)}\leq 1$  by the estimate ~$\cos(\pi v)^2 = \exp\{O_\ee(\abs{v}^2)\}$ for $|\Re(v)|\leq\frac12-\eps$ and~$\abs{\Im(v)}\leq 1$. If~$\abs{\Im(\alpha z)}>1$, then both the ratio~$\frac1{\alpha z-\alpha}$ and the integral are bounded, since~$\abs{\cos(\pi v)} \gg \e^{\pi|\Im(v)|}$, so that~\eqref{eq:lemmatan-integ} holds in this case as well.
\end{proof}

\begin{lemma}\label{mlftt}
  Let $\alpha\in\R$ with $0<|\alpha|<1/2$, and
  $$ D_\alpha := \{z\in\C \mid \tfrac{-1}{{3}|\alpha|} < \Re(z)<1\}. $$
  Then, taking the determination of the logarithm which is real on the real axis, the function given by
  \begin{align}\label{dfps}
    \psi_\alpha(z):=\log(1-\cot(\pi \alpha)\tan(\pi \alpha z))+\cot(\pi \alpha)\tan(\pi \alpha z)
  \end{align}
  is holomorphic on $D_\alpha$, where it satisfies
  \begin{align}
    &\psi_\alpha(z)\ll |z|^2+|z||\log(1-z)|, \qquad \Im(\psi_\alpha(z))\ll |z|^2,\label{dfps1}\\
    &\psi_\alpha(z)=\log(1-z)+z+O_\eps((|z|^2+|z|^3)|\alpha|^2) \quad \text{if $\abs{\alpha}\leq\tfrac{1}{2}-\eps$.}\label{dfps2}
  \end{align}
\end{lemma}

\begin{proof}
We can assume $\alpha\in(0,1/2)$, since $\psi_\alpha$ is even in $\alpha$.
For~$\abs{\Re(z)}<\frac1{2\alpha}$, let~$g_\alpha(z) := \cot(\pi \alpha) \tan(\pi \alpha z)$. Since, for~$z=x+iy$, we have
  \begin{align}\label{tsum}
    \tan(\pi \alpha z)=\frac{\tan(\pi \alpha x)(1-\tanh^2 (\pi \alpha y))+i\tanh(\pi\alpha y)(1+\tan^2 (\pi\alpha  x))}{1+\tan^2(\pi\alpha  x)\tanh^2(\pi\alpha y)},
  \end{align}
  we deduce, for~$x, y$ real and~$\abs{x}\leq 1$,
  \begin{align}
    \abs{\Re(g_\alpha(z))} \leq \abs{\cot(\pi\alpha)\tan(\pi \alpha x)} \leq \abs{x}\label{bfrg}.
  \end{align}
  by Lemma~\ref{tan_ineq}. Moreover, if~$x\leq 0$, then~$\Re(g_\alpha(z))\leq 0$. Using~\eqref{tsum} and the lower bound~\eqref{eq:tanineq-2}, we obtain~$\abs{\Re(g_\alpha(z))} \gg \cot(\pi \alpha) \tan(\pi \alpha \abs{x}) \gg_\eps \abs{x}$. Thus, there exists~$c_\eps>0$ such that
  \begin{equation}
    \Re(1-g_\alpha(z)) \geq \min(1-x, 1-c_\eps x). \label{eq:bfrg2}
  \end{equation}
  In particular, the function~$z\mapsto \psi_\alpha(z) = \log(1-g_\alpha(z)) + g_\alpha(z)$ is well-defined and holomorphic in~$D_\alpha$, and moreover~$\Im(\log(1-g_\alpha(z))) < \pi / 2$. 
  
  We observe that if $\alpha y\to\pm \infty$ (and so $z\to\infty$), then $g_{\alpha}(z)\sim\pm i\cot(\pi\alpha)$ uniformly in~$x$, so that $\psi_\alpha (z)\ll 1/\alpha\ll |z|$ and~\eqref{dfps1}-\eqref{dfps2} follow trivially. Thus, we can assume $\alpha y\ll1$.
  Also, by~\eqref{tsum} and Lemma~\ref{tan_ineq} we obtain
  \begin{align}
    |\Im(g_\alpha(z))| &\ll |\cot(\pi\alpha)|\big(|\tan (\pi x\alpha)|+|\tanh (\pi \alpha y)|)\ll |x|+|y|\ll |z|.\label{bfig}
  \end{align}
  In particular, recalling~\eqref{bfrg}, we have $g_\alpha(z)\ll |z|$.
  
  Now, assume~$\alpha \leq \frac12 - \eps$. By~\eqref{tsum}, and since $\alpha y\ll1$, we have 
  $$|\Im(g_{\alpha}(z))|\geq |\cot(\pi\alpha)\tanh(\pi \alpha y)|\gg_\eps |y| . $$ 
  By~\eqref{eq:bfrg2}, it then follows that $|1-g_{\alpha}(z)|\gg_\eps \min(1-x, 1-c_\eps x)+|y|\gg_\eps |1-z|$, and so $\abs{u_\alpha(z)}\gg_\eps1$ for $u_\alpha(z):=\frac{1-g_\alpha(z)}{1-z}$. Thus,
  \begin{align*}
    \log(u_\alpha(z))=u_\alpha(z)-1+O_\eps(|u_\alpha(z)-1|^2)=z-g_\alpha(z)+z (u_\alpha(z)-1)+O_\eps(|u_\alpha(z)-1|^2)
  \end{align*}
  and so~\eqref{dfps2} follows since $u_\alpha(z)-1=O_\eps(\alpha^2(|z|^2+|z|))$ by Lemma~\ref{tan_ineq2} (and since $\alpha z\ll1$). 
  
  We now move to~\eqref{dfps1}. 
  Since $g_\alpha(z)\ll |z|$, then if $|z|<\delta$ with $\delta>0$ sufficiently small, then by Taylor expansion one trivially has that~\eqref{dfps1} holds for $|z|<\delta$. By the assumption $|\alpha y|\ll1$ we have that~\eqref{dfps2}
  implies~\eqref{dfps1} for $|z|\geq \delta$ and $\alpha \leq \frac12 - \eps$. It follows that we can also assume $\alpha\gg1$ (and so also $\abs{z}\ll \alpha \abs{z}\ll1$).
  
  Finally, we observe that under the above assumptions we have that~\eqref{tsum} implies also 
  $$\Re(1-g_\alpha(z)) \geq  1-x (1-\tanh^2(\pi\alpha y))\gg 1-x+ y^2\gg |1-z|^2 \qquad (x \geq 0), $$
  which inequality~$\Re(1-g_\alpha(z)) \gg \abs{1-z}^2$ is also trivially true if~$x<0$. The bound~\eqref{dfps2} follows since $g_\alpha(z)\ll \abs{z}$, $\Im(\log(1-g_\alpha(z))) < \pi / 2$, and $\log(1-w)+w \ll \min\{\abs{w}^2, \abs{w}+\log\abs{1-w}\}$ for~$w\in\C\minus[1, \infty)$.
\end{proof}

\begin{lemma}\label{triab}
  Let $\gamma\leq \delta$ and let $a\in\N$ with $\gamma\leq a-\frac12\leq \delta$.
  Let $g(z)$ be holomorphic on a neighborhood of $\gamma\leq \Re(z)\leq \delta$, where it satisfies $|g(z)|\leq C_1 |z|^m+C_2$ for some $m,C_1,C_2\geq0$.
  Then, with~$z=r$, we have
  \begin{align*}
    \sum_{n=a}^r g(n) = \Oh_{(g,a)}^r(C_1((\abs{a}+\abs{z})^{m+1}+1)+C_2(\abs{a}+\abs{z}+1)), && (r\in\Z\cap [a, \delta])
  \end{align*}
  using the notation~\eqref{eq:def-Oh} with~$D_2 = \{z,\ \gamma+\frac12 \leq \Re(z)\leq \delta\}$.
\end{lemma}
\begin{proof}
  We apply Lemma~\ref{lem:abelplana} with $\alpha=a-\frac12$, $\beta=r-\frac12$, we find
  \begin{align*}
    \sum_{n=a}^{r-1} g(n)=\int_{a-\frac12}^{r-\frac12}g(z)\df z- i \int_{0}^\infty  \frac{g(a-\frac12 + it)-g(a-\frac12 - it)-g(r-\frac12+ it)+g(r-\frac12 - it)}{\e^{2\pi t}+1}\df{t},
  \end{align*}
    Denoting by $\varphi(r)-g(r)$ the right hand side, we immediately see that $\varphi$ extend to a holomorphic function in $\gamma+\frac12 \leq \Re(z)\leq \delta$ and that $\varphi(z)\ll C_1((\abs{a}+\abs{z})^{m+1}+1)+C_2(\abs{a}+\abs{z}+1)$ in this strip.
\end{proof}

\begin{lemma}\label{madc}
  Let $(h,k)=1$, $3\leq h<k$ and $r_0\in\Z/h\Z$. Then for all $0\leq r<k$ with~$r\equiv r_0\mod{h}$, writing~$z=r/k$, we have
  \begin{align*}
    \ssum{0\leq n\leq r\\ h\nmid n}\bigg(\log \bigg(1&-\cot\Big( \frac{\pi n\overline k }{h} \Big)\tan\Big( \frac { \pi n}{hk} \Big)\bigg)+\cot\Big( \frac{\pi n\overline k }{h} \Big)\tan\Big( \frac { \pi n}{hk}\Big)\bigg)\\[-1em]
    {}& =\frac kh\Big(\frac{\pi i (z^2-z)}{2}+\Lie(z)+\frac{\pi i }{12}+z \log(2 \pi z/e)\Big) \\
    {}& \qquad\qquad +\Oh_{(h,k,r_0)}^{z}((1+\abs{\log(1-z)} + |z|^4)(1+k/h^2 + \log(k/h))),
  \end{align*}
  using the notation~\eqref{eq:def-Oh} with~$D_2 = \{z\in\C, 0\leq \Re(z) < 1\}$. Moreover, the error term~$\abs{\log(1-z)}$ can be omitted if~$z \in [0, 1 - \frac hk(1-\{\frac{r_0-k}h\})]$.
\end{lemma}
\begin{proof}
  In this proof, the notation~$\Oh$ will stand for~$\Oh_{(k,h,r_0)}^z$ relative to the set~$D_2 = \{z, \Re(z)\in[0, 1)\}$. We divide the sum into congruence classes $n\equiv \ell k\mod h$, $|\ell|<h/2$, where the possible term $\ell=h/2$ is excluded since the summand is zero in this case. With the notation~\eqref{dfps}, we write the sum to be computed as
  $$ \sum_{0 < \abs{\ell} < h/2} S(\ell), \qquad S(\ell) := \ssum{0<n\leq r \\ n\equiv \ell k \mod{h}} \psi_{\ell/h}\Big(\frac{n}{k\ell}\Big). $$
  We consider each~$\ell$ separately. 
  We first consider~$\ell \neq 1$. Let~$\delta_\ell := \{\frac{r_0-k\ell}{h}\}$, and
  $$ \alpha_\ell = \left\lceil\frac{-k\ell}h\right\rceil - \frac12, \qquad \beta_\ell(r) := \floor{\frac{r-k\ell}{h}} + \frac12 = \frac{r-k\ell}{h} + \frac12 - \delta_{\ell}. $$
  Then by Lemma~\ref{lem:abelplana}, we have
  $$ S(\ell) = \sum_{\alpha_\ell < m \leq \beta_\ell(r)} \psi_{\ell/h}(1 + \tfrac{hm}{k\ell}) = I + C(\alpha_\ell) - C(\beta_\ell(r)), $$
  where
  \begin{align*}
    I := \int_{\alpha_\ell}^{\beta_\ell(r)} \psi_{\ell/h}(1 + \tfrac{ht}{k\ell})\df t, \qquad  C(\gamma) := 2 \int_0^\infty \frac{\Im (\psi_{\ell/h}(1+\frac{h}{k\ell}(\gamma + it)))\df t}{\e^{2\pi t} + 1}.
  \end{align*}
Note that~$\beta_\ell(r) = \frac{r-r_0}h + \beta_\ell(r_0)$, the right-hand side of which depends holomorphically on~$r$.
Splitting the integral as $ \int_{- {k\ell}/h}^{(z-\ell) k/h}+\int_{\alpha_\ell}^{- {k\ell}/h}+\int_{(z-\ell) k/h}^{\beta_\ell(r)}$ with~$z = r/k$, we get
  \begin{align*}
    I &{} =
        \frac{k}h \int_0^{z} \psi_{\ell/h}(t/\ell)\df t + \int_{\alpha_\ell+\frac{k\ell}h}^{0}\psi_{\ell/h}(\tfrac{ht}{k\ell})\df t + \int_0^{\frac12-\delta_\ell} \psi_{\ell/h}(\tfrac z\ell+\tfrac{ht}{k\ell})\df t \\
    &{} =  \frac{k}h \int_0^{z} \psi_{\ell/h}(t/\ell)\df t + \Oh\big(\tfrac{1+\abs{z}^2}{\ell^2}\big),
  \end{align*}
  by Lemma~\ref{mlftt}.

  Next, we have
  $$ C(\beta_\ell(r)) = \frac1{2i}\sum_\pm \pm \int_0^\infty \psi_{\ell/h}\Big(\frac z\ell + \frac{h}{k\ell}\Big(\frac12-\delta_\ell\Big) \pm i\frac{t h}{k\ell}\Big)\frac{\df t}{\e^{2\pi t}+1}. $$
  This also defines a holomorphic function of~$z$ for~$\Re(z) \in [0, 1)$, by Lemma~\ref{mlftt}. Since~$\ell \neq 1$, it is bounded by~$\Oh(\frac{\abs{z}^2+1}{\ell^2})$.
  
  Grouping the above discussion, we deduce for~$\ell\neq 1$ the estimate
  \begin{equation}
    S(\ell) = \frac{k}{h} \int_0^z \psi_{\ell/h}(t/\ell)\df t + \Oh\Big(\frac{1+\abs{z}^2}{\ell^2}\Big).\label{eq:estim-Sell-neq1}
  \end{equation}

  Consider now the case~$\ell = 1$. We recall the notation~$\delta_\ell$ from above. Since~$h\geq 3$, by~\eqref{dfps2} and Lemma~\ref{triab} we have
  \begin{equation}
    S(1) = \Oh(\tfrac{k}{h^3}(\abs{z}^4+1)) + \ssum{0<n\leq r \\ n\equiv k \mod{h}} \Big( \frac nk + \log\Big(1 - \frac nk\Big)\Big).\label{eq:sumcotan-ell-1}
  \end{equation}
  Let~$a\in\{1, \dotsc, h\}$ satisfy~$a\equiv k\mod{h}$, $q = \frac{k-a}h$ and~$g = \floor{\frac{r-a}h} = \frac{r-a}h - \delta_1$. In the sum, the integer~$m = \frac{k-n}{h}$ runs through~$\Z\cap[q-g, q]$, so that
  \begin{align*}
    \ssum{0<n\leq r \\ n\equiv k \mod{h}} \log\Big(1 - \frac nk\Big) = {}& (g+1)\log\Big(\frac hk\Big) + \log\frac{\Gamma(q+1)}{\Gamma(q-g)} \\[-1em]
    = {}& \int_0^{g+1}\Big(\log\Big(\frac hk\Big) + \frac{\Gamma'}{\Gamma}(q+1-v)\Big)\df v \\
    = {}& \int_0^{g+1}\Big(\log\Big(\frac hk(q+1-v)\Big) - \frac1{q+1-v} - \int_0^\infty \frac{\{s\}\df s}{(s+q+1-v)^2}\Big)\df v
  \end{align*}
  by Stirling's formula~\cite[Theorem~II.0.12]{Tenenbaum2015a}. First, setting~$t = \frac{(v-1)h+a}k$, we have
  \begin{align*}
    \int_0^{g+1}\log\Big(\frac hk(q+1-v)\Big)\df v = {}& \frac kh\int_0^z \log(1-t)\df t + \frac kh\Big[\int_{\frac{a-h}{k}}^0 + \int^{z-\frac hk \delta_1}_z\Big] \log(1-t)\df t \\
    = {}& \frac kh\int_0^z \log(1-t)\df t + \Oh(\abs{\log(1-z)} + \log(1+k/h)).
  \end{align*}
  Secondly, we have
  \begin{align*}
    \int_0^{g+1} \frac{\df v}{q+1-v} = \log\Big(1-z+\frac{h}k\delta_1\Big) - \log\Big( 1 + \frac{h-a}k\Big) = \Oh(\abs{\log(1-z)} + \log(1+k/h)).
  \end{align*}
  Note that in both cases, as well as in the following computations, the error term~$\abs{\log(1-z)}$ can be omitted if~$z \in [0, 1 - \frac hk(1-\delta_1)]$ (which is the case when~$z=r/k$).
  Finally,
  \begin{align*}
    \int_0^{g+1}\int_0^\infty \frac{\{s\}\df s}{(s+q+1-v)^2}\df v = {}& \int_0^\infty \frac{\{s\}(g+1)\df s}{(s+q-g)(s+q+1)} \\
    = {}& \int_0^\infty \{s\}\frac{1 + \frac kh z - \frac ah - \delta_1}{(s+\frac kh(1-z) + \delta_1)(s+\frac{k-a}h+1)}\df s \\
    = {}& \Oh(\log (1+k/h)+\log(1+|z|) ).
  \end{align*}
  We turn to the contribution of the term~$n/k$ in~\eqref{eq:sumcotan-ell-1}. By a direct computation, we find
  \begin{align*}
    \ssum{0<n\leq r \\ n\equiv k \mod{h}} \frac nk - \frac kh \int_0^z t\df t =
    {}& z(\tfrac12-\delta_1) + 1 - \tfrac{h}{2k}(\tfrac ah + \delta_1)(\tfrac ah + 1 - \delta_1) = \Oh(\abs{z}+1).
  \end{align*}
  On the other hand, we note that by~\eqref{dfps2},
  $$ \int_0^z (t + \log(1-t))\df t = \int_0^z \psi_{1/h}(t)\df t + \Oh\Big(\frac{1+\abs{z}^4}{h^2}\Big). $$
  Grouping the above estimates, we deduce
  \begin{equation}
    S(1) = \frac kh\int_0^z \psi_{1/h}(t)\df t + \Oh\big((1 + \tfrac{k}{h^3} + \log(k/h))(1 + \abs{z}^4 + \abs{\log(1-z)})\big).\label{eq:estim-S1}
  \end{equation}
  
  We now sum the estimates~\eqref{eq:estim-Sell-neq1}, \eqref{eq:estim-S1} over~$\ell$, getting
  $$ \sum_{0<\abs{\ell}<h/2} S(\ell) = \frac{k}{h} \int_0^z \sum_{0<\abs{\ell}<h^{2/3}} \psi_{\ell/h}(t/\ell)\df t +
  \Oh((1 + \tfrac{k}{h^{3}} + \log(k/h))(1+\abs{z}^4+\abs{\log(1-z)})). $$
  The main term is evaluated by~\eqref{dfps2} (for~$\abs{\ell}\leq h/3$) and~\eqref{dfps1} (for~$h/3<\abs{\ell}<h/2$) as
  \begin{align*}
    \sum_{ 0<|\ell|\leq h/2}\int_0^z \psi_{\ell/h}(t/\ell)\df t&=  \sum_{ 0<|\ell|\leq h/3}\bigg(\int_0^z(\log(1-t/\ell)+t/\ell)\df t+ \Oh\Big(\frac{1+\abs{z}^4}{h^2}\Big)\bigg) + \Oh\Big(\frac{1+\abs{z}^3}{h}\Big) \\
    & =  \sum_{0<\ell<h/3}\int_0^z \log(1-t^2/\ell^2)\df t+ \Oh\Big(\frac{1+\abs{z}^4}{h}\Big) \\
    & =  \sum_{\ell>0}\int_0^z\log(1-t^2/\ell^2)\df t + \Oh\Big(\frac{1+\abs{z}^4}{h}\Big).
  \end{align*}
  For $-1\leq t\leq1$ we have $\prod_{\ell=1}^\infty (1-t^2/\ell^2)=\frac{\sin (\pi t)}{\pi t}$. Moreover, for $0\leq z<1$ we have
  \begin{align*}
    \int_{0}^{z}\log(2\sin(\pi t))\df t&=\int_{0}^{z}\log(-i\e^{\pi i t}(1-\e^{-2\pi i t}))\df t=\frac{\pi i (z^2-z)}{2}+\int_{0}^{z}\log(1-\e^{-2\pi i t})\df t\\
    &=\frac{\pi i (z^2-z)}{2}+\Lie(z)+\frac{\pi i }{12},
  \end{align*}
  by~\eqref{dig}. Collecting the previous estimates, we conclude that
  \begin{align*}
    \sum_{0 < \abs{\ell} < h/2} S(\ell) = {}& \frac kh\Big(\frac{\pi i (z^2-z)}{2}+\Lie(z)+\frac{\pi i }{12}+z \log(2 \pi z/\e)\Big) \\
    \qquad & + \Oh\Big(\Big(1 + \frac k{h^2} + \log(k/h)\Big)(1 + \abs{z}^4 + \abs{\log(1-z)}) \Big),
  \end{align*}
  as claimed.
\end{proof}

\subsection{Second reciprocity formula for the $q$-Pochhammer symbol}\label{sec:recip2-poch}

\begin{theorem}\label{thp}
  Let $4\leq h< k$ with $(h,k)=1$ and $r_0\in \{0, \dotsc, h-1\}$. Let~$0\leq r<k$ with $r\equiv r_0\mod{h}$.
  We have 
  \begin{align*}
    \frac{(\e[-{\overline h}/k])_{r}\e[-\frac{h}{24k}]}{(\e[{\overline k}/h])_{r_0}\e[\frac{k}{24h}]}=\exp\bigg(&\frac{k}h\Lie(r/k) -\frac{\pi }k\ssum{1\leq n\leq r_0}\cot\Big(\pi\frac{n\overline k }h\Big)\frac{ n}{h}+\frac {\pi}k\floor{\frac rh}c_0\Big(\frac{\overline k}h\Big) \\
    {}& + \Oh_{h,k,r_0}^{z}\Big((1+|z|^4+\abs{\log(z(1-z))})\Big(1+\log\Big(\frac kh\Big)+\frac k{h^{2}}\Big)\Big)\bigg)
  \end{align*}
  where~$z=r/k$, using the notation~\eqref{eq:def-Oh} with domain~$D_2 = \{z\neq 0, \Re(z)\in[0, 1)\}$. Moreover, the term~$\abs{\log(z(1-z))}$ can be omitted if~$z\in[\frac{r_0}k, 1-\frac{h}k(1-\{\frac{r_0-k}{h}\})]$.
\end{theorem}
\begin{proof} In this proof, the notation~$\Oh$ will stand for~$\Oh_{h,k,r_0}^z$ with respect to the domain~$D_2 = \{z\neq 0, \Re(z)\in[0, 1)\}$. 
  Applying the identity $\frac{\overline h}{k}+\frac{\overline k}{h}\equiv \frac{1}{hk}\mod 1$ we have
  \begin{align*}
    (\e[-{\overline h}/k])_{r}&=\prod_{n=1}^r\bigg(1-\e[\frac{n\overline k }h - \frac n{hk}]\bigg)
    = \mathcal{P}\cdot \pprod{n=1\\ h\nmid n}^r\bigg(1-\e[\frac{n\overline k }h - \frac { n}{hk}]\bigg)
  \end{align*}
  where $\mathcal{P}:=\prod_{n=1}^{\lfloor r/h\rfloor}\big(1-\e[-\frac { n}{k}]\big)$. Note that for $x,y\in\C$, we have
  \begin{align*}
    1-\e[y-x]&= \frac{1}2 (1+\e[-x])\big (1-\e[y]\big)(1-\tan (\pi x) \cot (\pi y)).
  \end{align*}
  Thus, since $\prod_{n=1}^{h-1}(1-\e[\frac{n\overline k }h])=h$, we have
  \begin{align}\label{feqr2}
    (\e[-{\overline h}/k])_{r}&=(\e[{\overline k}/h])_{r_0}\, h^{\floor{ r/h}} \cdot \mathcal{P}\cdot \mathcal{M}\cdot \mathcal{L},
  \end{align}
  where 
  \begin{align*}
    \mathcal{M}=\pprod{n=1\\ h\nmid n}^r\frac{1+\e[-\frac{ n}{hk}] }2, &&   \mathcal{L}    =   \pprod{n=1\\ h\nmid n}^r\bigg(1-\cot\Big( \frac{\pi n\overline k }{h} \Big)\tan\Big( \frac { \pi n}{hk} \Big)\bigg).
  \end{align*}
  First, we examine $\mathcal{M}$. We have
  \begin{align*}
    \mathcal{M}&=\exp\bigg(-\ssum{n=1\\ h\nmid n}^r\frac{\pi i n}{hk}+\ssum{n=1\\ h\nmid n}^r\log\Big(\cos\Big(\frac{\pi n}{hk}\Big)\Big)\bigg).
  \end{align*}
  We split the second sum as $\ssum{1\leq n\leq r}-\ssum{1\leq n\leq r-r_0,\, h|n}$. For any $\eps>0$ we have $\log(\cos(\frac{\pi w}{hk}))\ll_\eps (\frac {|w|}{hk})^2$ for $\abs{\Re(w/(hk))}\leq (1-\eps)/2$ and so applying Lemma~\ref{triab} we obtain, with~$z = r/k$,
  \begin{align}\label{bfmm}
    \mathcal{M}&=\exp\bigg(-\ssum{n=1\\ h\nmid n}^r\frac{\pi i n}{hk} + \frac{k}{h^2} \Oh(\abs{z}^3+1) \Big)\bigg)
    = \exp\bigg(-\frac{\pi i (r^2+r)}{2hk}+\frac{k}{h^2}\Oh(\abs{z}^3+1)\bigg).
  \end{align}
  
  We then move to $\mathcal{P}=\exp(\log \mathcal{P})$. Taking the determination which is real on the negative imaginary axis, we have that $\log(\frac{1-\e[-w]}{2\pi iw})$ is holomorphic and $O_\eps(|w|)$ for $\abs{\Re(w)}<1-\eps$.
Thus, by Lemma~\ref{triab},
  \begin{align*}
    \mathcal{P}&=\exp\bigg(\sum_{1\leq n \leq (r-r_0)/h} \log\Big(2\pi i\frac{n}{k}\Big)+\frac{k}{h^2}\Oh(\abs{z}^2+1)\bigg)\\
    &=\exp\bigg( \log\Gamma\Big(1+\frac{r-r_0}h\Big)+\frac{r-r_0}h\log\frac{2\pi i}k+\frac{k}{h^2}\Oh(\abs{z}^2+1)\bigg).
  \end{align*}
  Write~$\log (\Gamma(1+w))=(w+1/2)\log(w+1)-w+\E_1(w)$ with $\E_1$ holomorphic and $O(|\log (w+1)|)$ on $\Re(w)>-1$. Abbreviating temporarily~$q = \frac{r-r_0}h$, it follows that 
  \begin{align}
    \mathcal{P}&=\exp\bigg(
    \Big(\frac12+q\Big)\log(1+q)+ q\log\frac{2\pi i}{ke}
    +\E_1(1+q)+\frac{k}{h^2}\Oh(\abs{z}^2+1)
    \bigg)\notag\\
    &=\exp\bigg( {\frac rh}\log\Big(\frac{2\pi i  r}{ke}\Big)- \Big\lfloor \frac{r}h\Big\rfloor\log h+\Oh((1+\log(k/h)+\tfrac{k}{h^2})(1+\abs{z}^2+\abs{\log(z)})
    \bigg),\label{bfpp}
  \end{align}
  since
  \begin{align*}
    &\Big(\frac12-\frac{r_0}{h}\Big)\log\Big(\frac{kz+h-r_0}{h}\Big)+\frac{kz}h\log\Big(1+\frac{h-r_0}{kz}\Big)+{}\\
    &\qquad- \frac{r_0}h\log\frac{2\pi i h}{ke}+\E_1\Big(\frac{kz+h-r_0}h\Big) = \Oh(\abs{z} + \abs{\log(z)}).
  \end{align*}
  Note that the terms~$\abs{\log(z)}$ can be omitted if~$z\in[\frac{r_0}k, 1)$.

  It remains to study $\mathcal{L}$. By Lemma~\ref{tan_ineq} we have $|\tan( \frac{\pi n}{hk})\cot(\pi \frac{n\overline k}h)|<1$ and so we can write $\mathcal{L}=\exp(\log \mathcal{L})$ with the principal determination. First, we consider
  \begin{equation*}
    \ssum{1\leq n\leq r\\ h\nmid n}\tan\Big(\pi\frac n{hk}\Big)\cot\Big(\pi\frac{n\overline k }h\Big)=\ssum{1\leq n\leq r\\ h\nmid n}\frac {\pi n}{hk}\cot\Big(\pi\frac{n\overline k }h\Big)+\ssum{1\leq n\leq r\\ h\nmid n}\E_2\Big(\frac n{hk}\Big)\cot\Big(\pi\frac{n\overline k }h\Big),
  \end{equation*}
  where $\E_2(z):=\tan(\pi z)-\pi z$. Clearly, $\E_2(z)$ is holomorphic and $O_\eps(|z|^2)$ in $|\Re(z)|<(1-\eps)/2$. Thus, dividing in congruence classes modulo $h$, the second summand above is
  \begin{align*}
    \Oh\Big((1+|z|^3)\sum_{\ell=1}^h\frac 1\ell\frac k{h^2}\Big) = \Oh\Big((1+|z|^3)\frac {k\log h}{h^2}\Big).
  \end{align*}
  Also,
  \begin{align*}
    \ssum{1\leq n\leq r\\ h\nmid n}\cot\Big(\pi\frac{n\overline k }h\Big)\frac{\pi n}{hk}&=\ssum{1\leq n\leq r_0}\cot\Big(\pi\frac{n\overline k }h\Big)\frac{\pi n}{hk}+\sum_{\ell=1}^{\floor{r/h}}\ssum{1\leq n<h}\cot\Big(\pi\frac{n\overline k }h\Big)\frac{\pi (n+\ell h)}{hk}\\
    &=\frac{\pi }k\ssum{1\leq n\leq r_0}\cot\Big(\pi\frac{n\overline k }h\Big)\frac{ n}{h}-\floor{\frac rh}\frac {\pi}kc_0\Big(\frac{\overline k}h\Big)
  \end{align*}
  since $\sum_{n=1}^{h-1}\cot(\pi\frac{n}h)=0$. Thus, by Lemma~\ref{madc} we have 
  \begin{align*}
    \mathcal{L}=\exp\bigg(&\frac{\pi i (r^2-rk)}{2kh}+\frac kh\Big(\Lie(r/k)+\frac{\pi i }{12}+\frac rk \log\Big( \frac {2 \pi r}{ke}\Big)\Big)+{}\\
    &+\frac{\pi }k\ssum{1\leq n\leq r_0}\cot\Big(\pi\frac{n\overline k }h\Big)\frac{ n}{h}    -\frac {\pi}k\floor{\frac rh}c_0\Big(\frac{\overline k}h\Big) \\
    & +\Oh\Big(\Big(1+\frac{k}{h^2}+\log(k/h)\Big)(1+|z|^4+|\log(1-z)|)\Big)\bigg),
  \end{align*}
  where the error term~$\abs{\log(1-z)}$ can be omitted if~$z\in[0,1 - \frac{h}k(1-\{\frac{r_0-k}h\})]$.
  The theorem then follows by~\eqref{feqr2},~\eqref{bfmm} and~\eqref{bfpp}, since $\frac{ h}{24k}=O(1).$
\end{proof}

\section{Proof of Theorem~\ref{th:1}}\label{pmt1}

Throughout the rest of the section,  $K$ will denote any hyperbolic knot $K\neq 7_2$ with at most $7$ crossings. 

We will use the same notation as in Section~\ref{ss:1rf}. In particular, all error terms will be allowed to depend on~$d$ and $\gamma$.

For~$n\in\N_{\geq 0}$ and~$\alpha=\frac hk\in\Q$, with $(h,k)=1$, we let
\begin{align}\label{sb}
  [\alpha]_n :=k^{-1/2}  (\e[\alpha])_{n'}, 
\end{align}
where $n'\equiv n\mod {k}$, $0\leq n'<k$.

There exist $m,m_1,\dots,m_4\in\N$, $\iota,\upsilon\in\Z$ and linear functions  $\ell_{i,j}(\boldsymbol r)=\sum_{u=1}^m\kappa_{i,j}(u)r_u$ with $\kappa_{i,j}\in\{0,\pm1\}$ such that 
\begin{align}
  \J_{K}( x) ={}& \denom(x)^{\iota}\e[\upsilon x] \sumstar_{0\leq r_1,\dots, r_m<k}\frac{\prod_{j=1}^{m_1}[ x]_{\ell_{1,j}(\boldsymbol r)}\prod_{j=1}^{m_2}\overline{[ x]}_{\ell_{2,i}(\boldsymbol r)}}{\prod_{j=1}^{m_3}[ x]_{\ell_{3,j}(\boldsymbol r)}\prod_{j=1}^{m_4}\overline{[ x]}_{\ell_{4,j}(\boldsymbol r)}} \label{exja} \\
  = {}& \denom(x)^{\iota}\e[\upsilon x]  \sumstar_{0\leq r_1,\dots, r_m<k} \Pi_K\big([x]_{\ell_{i,j}(\boldsymbol r)}\big), \notag
\end{align}
where $\sumstar$ indicates that the sum is restricted to the terms with $0\leq \ell_{i,j}(\bs r)<k$ and, here and what follows, we put
$$ \Pi_K\big(z_{i, j}\big) = \frac{\prod_{j=1}^{m_1} z_{1,j} \prod_{j=1}^{m_2} \overline{z_{2,j}}}{\prod_{j=1}^{m_3} z_{3,j} \prod_{j=1}^{m_4} \overline{z_{4,j}}}. $$
The Kashaev invariants for the knots under consideration has been given for example  in~\cite{Ohtsuki52,OhtsukiYokota6,Ohtsuki7}. In all these cases, $m+3$ coincides with the number of crossings of $K$, moreover
$$ \iota=\frac{3-m}2, \qquad m_1+m_2+m_3+m_4 = 3m-1, $$
and the values of $m_i,\ell_{i,j}$ are as in Figure~\ref{fig:table-K-param}, where we used the formula $ [\alpha]_n [\overline{\alpha}]_{\denom({\alpha})-n}=1$ to write the Kashaev invariants given in~\cite{Ohtsuki52,OhtsukiYokota6,Ohtsuki7} as in~\eqref{exja}. Finally, different variants of the definition of the Kashaev invariant lead to slightly different values for $\upsilon$~(cf.~\cite[p.677 footnote 4]{Ohtsuki52} and~\cite{Yokota2003}). 
In the context of the modularity conjecture it is natural to always take $\upsilon=0$, which we shall do in the following. This choice will lead to the expression~\eqref{eq:J0-approxmod} for the reciprocity formula, as conjectured by Zagier. Using~\eqref{expid} one can then easily deduce the suitable modified reciprocity formula corresponding to other choices of $\upsilon$.

\begin{figure}[h]
  \centering
  \begin{equation*}
    \begin{array}{|c|c|l|l|}
      \hline
      K & (m_j)_{1 \leq j \leq 4} & \ell_{i,j} = \ell_{i,j}(\br) \\ \hline

      4_1 &
      \begin{array}{l}
        (1, 1, 0, 0)
      \end{array} & 
      \begin{array}{l}
        \ell_{1,1} = \ell_{2,1} = r
      \end{array} \\ \hline

      5_2 &
      \begin{array}{l}
        (0, 1, 2, 2)
      \end{array} & 
      \begin{array}{l}
        \ell_{2,1}=r_1+r_2,\quad \ell_{3,1}=r_1+r_2 \\
        \ell_{3,2}=\ell_{4,1}=r_2,\quad \ell_{4,2}=r_1 \\
      \end{array} \\ \hline

      6_1 &
      \begin{array}{l}
        (0, 2, 3, 3)
      \end{array} & 
      \begin{array}{l}
        \ell_{2,1}=r_1+r_2,\quad \ell_{2,2}=r_1+r_2+r_3 \\
        \ell_{3,1}=r_1,\quad \ell_{3,2}=r_1+r_2,\quad \ell_{3,3}=r_1+r_2+r_3 \\
        \ell_{4,1}=r_1,\quad \ell_{4,2}=r_2,\quad \ell_{4,3}=r_3
      \end{array} \\ \hline

      6_2 &
      \begin{array}{l}
        (2, 1, 2, 3)
      \end{array} & 
      \begin{array}{l}
        \ell_{1,1}=r_1,\quad \ell_{1,2}=r_2+r_3,\quad \ell_{2,1}=r_1 \\
        \ell_{3,1}=\ell_{4,1}=r_2,\quad \ell_{3,2}=r_3,\quad \ell_{4,2}=r_1-r_2,\quad \ell_{4,3}=r_2+r_3
      \end{array} \\ \hline

      6_3 &
      \begin{array}{l}
        (1, 1, 3, 3)
      \end{array} & 
      \begin{array}{l}
        \ell_{1,1}=\ell_{2,1}=r_2,\quad \ell_{3,1}=\ell_{4,1}=r_1 \\
        \ell_{3,2}=\ell_{4,2}=r_3,\quad \ell_{3,3}=r_2-r_3,\quad \ell_{4,3}=r_2-r_1
      \end{array} \\ \hline

      7_3 &
      \begin{array}{l}
        (2, 3, 3, 3)
      \end{array} & 
      \begin{array}{l}
        \ell_{1,1}=\ell_{2,1}=r_2,\quad \ell_{1,2}=r_2-r_1 \\
        \ell_{2,2}=\ell_{3,1}=r_2-r_3,\quad \ell_{2,3}= \ell_{3,2}=r_2-r_3-r_4,\quad \ell_{3,3}=r_1 \\
        \ell_{4,1}=r_2-r_1,\quad \ell_{4,2}=r_3,\quad \ell_{4,3}=r_4
      \end{array} \\ \hline

      7_4 &
      \begin{array}{l}
        (3, 0, 4, 4)
      \end{array} & 
      \begin{array}{l}
        \ell_{1,1}=\ell_{4,1}=r_1+r_2,\quad \ell_{1,2}=r_2+r_3 \\
        \ell_{1,3}=\ell_{4,2}=r_3+r_4 \\
        \ell_{3,1}=r_1,\quad \ell_{3,2}=\ell_{4,3}=r_2 \\
        \ell_{3,3}=\ell_{4,4}=r_3,\quad \ell_{3,4}=r_4
      \end{array} \\ \hline

      7_5 &
      \begin{array}{l}
        (2, 2, 3, 4)
      \end{array} & 
      \begin{array}{l}
        \ell_{1,1}=\ell_{2,1}=r_3,\quad \ell_{1,2}=\ell_{4,1}=r_3-r_4 \\
        \ell_{2,2}=\ell_{3,1}=r_2,\quad \ell_{3,2}=\ell_{4,2}=r_1,\quad \ell_{3,3}=r_4 \\
        \ell_{4,3}=r_2-r_1 \\
        \ell_{4,4}=r_3-r_2
      \end{array} \\ \hline

      7_6 &
      \begin{array}{l}
        (2, 1, 4, 4)
      \end{array} & 
      \begin{array}{l}
        \ell_{1,1}=\ell_{4,1}=r_2,\quad \ell_{1,2}=r_3+r_4 \\
        \ell_{2,1}=r_2+r_3,\quad \ell_{3,1}=\ell_{4,2}=r_1,\quad \ell_{3,2}=r_2-r_1 \\
        \ell_{3,3}=\ell_{4,3}=r_3,\quad \ell_{3,4}=\ell_{4,4}=r_4.        
      \end{array} \\ \hline

      7_7 &
      \begin{array}{l}
        (2, 1, 4, 4)
      \end{array} & 
      \begin{array}{l}
        \ell_{1,1}=r_1+r_2, \quad \ell_{1,2}=r_3+r_4 \\
        \ell_{2,1}=r_2+r_3, \quad \forall j\in\{1, 2, 3, 4\},\ \ell_{3,j}=\ell_{4,j}=r_j
      \end{array} \\ \hline
    \end{array}
  \end{equation*}
  \caption{Parameters of the Kashaev invariants}
  \label{fig:table-K-param}
\end{figure}

We divide the sum over $\br$ restricting the $r_i$ into congruence classes modulo $q$ and in intervals of length $k/d$:
\begin{align}\label{splj}
  \J_{K}(\gamma x) = \denom(\gamma x)^{\frac{3-m}2} \sum_{\bs L\in\{0,\dots,d-1\}^m}\,\sum_{\bs s\mod q} \J_K(\gamma, x; \boldsymbol L,\boldsymbol s)
\end{align}
where, for $\boldsymbol L=(L_1,\dots,L_m)\in\Z_{\geq 0}^m$, $\boldsymbol s=(s_1,\dots,s_m)\in(\Z/q\Z)^m$ we write
\begin{align*}
  \J_K(\gamma, x; \bL, \bss) =   \ssumstar{0\leq r_1,\dots, r_m<k,\\  \floor{r_i d/k}=L_i,\ \forall i\\ r_i\equiv s_i \mod{q},\ \forall i}
  \Pi_K\big([\gamma x]_{\ell_{i,j}(\boldsymbol r)}\big).
\end{align*}
By Theorem~\ref{th:Ir} and Remark~\ref{remplr}, we have
$$ [\gamma x]_{r} = [x]_{\lfloor r d/k\rfloor} A(p,q) \Phi_{r}(\{r d/k\}), $$
where for $s\in\Z/q\Z$, $\lambda\in[0,1)$, we define
\begin{align} 
  A(p,q) &= \e\Big( - \frac{\deks(p,q)}2 - \frac 18\Big), \notag\\
  \Phi_{s}(\lambda) &= \exp\bigg(\frac k{qd} \Big(\Lie(\lambda) +\frac{\pi i }{12}\Big)+\frac{\pi i d}{12kq}+ \Er_s(\lambda,d/k) \bigg). \label{defb}
\end{align}
It follows that
\begin{align*}
  \J_K(\gamma, x; \bL, \bss) =   A(p,q)^{m_1+m_4-m_2-m_3} \ssumstar{0\leq r_1,\dots, r_m<k,\\  \floor{r_i d/k}=L_i,\ \forall i\\ r_i\equiv s_i \mod{q},\ \forall i}
  \Pi_K\big([x]_{\lfloor \ell_{i,j}(\boldsymbol r) d/k\rfloor} \Phi_{\ell_{i,j}(\boldsymbol s)}(\{\ell_{i,j}(\boldsymbol r) d/k\})\big).
\end{align*}
Now, let $\lambda_i=\{r_i d/k\}$.
The next lemma shows that the contribution of the terms for which $\ell_{i,j}(\bs \lambda)\notin[0,1)$ for some $i,j$ is negligible.

\begin{lemma}\label{extra_terms}
  There exists $\delta>0$ such that
  \begin{align*}
    \bigg|\Pi_K\big(\exp\big(\Lie(\ell_{i,j}(\bs \lambda))\big) \big)\bigg| \leq  \exp\Big(\frac{\Vol(K)}{2\pi}-\delta\Big)
  \end{align*}
  whenever $\ell_{i,j}(\bs \lambda)\notin[0,1)$ for some $i,j$.
\end{lemma}

We postpone the proof of Lemma~\ref{extra_terms} to Section~\ref{prolex}. Since $[x]_{r}\asymp 1$ for $0\leq r<d$, applying~\eqref{bdfg}
and the above lemma we obtain
\begin{equation}
  \begin{split}\label{reme}
    \J_K(\gamma, x ; \boldsymbol L,\boldsymbol s) &=  A(p,q)^{m_1+m_4-m_2-m_3}
    \Pi_K\big([ x]_{ \ell_{i,j}(\boldsymbol L) }\big)
    \J^*_K(\gamma, x; \boldsymbol L,\boldsymbol s)+{}\\
    &\quad+O\bigg(k^{O(1)}\exp\Big(\Big(\frac{\Vol(K)}{2\pi}-\delta\Big)\frac{k}{qd}\Big)\bigg),
  \end{split}
\end{equation}
where
\begin{align*}
  \J^*_K(\gamma, x; \boldsymbol L,\boldsymbol s) =    \ssumstar{0\leq r_1,\dots, r_m<k,\\  \floor{r_i d/k}=L_i,\ \forall i\\ r_i\equiv s_i \mod{q},\ \forall i,\\ 
    \ell_{i,j}(\bs \lambda)\in[0,1)\ \forall i,j }
  \Pi_K\big(\Phi_{\ell_{i,j}(\boldsymbol s)}(\{\ell_{i,j}(\boldsymbol r) d/k\})\big).
\end{align*}
We notice that the condition $0\leq \ell_{i,j}(\br)<k$, which is implicit in the summation $\sum^*$, can be written as $0\leq \ell_{i,j}(\bs L)+\ell_{i,j}(\bs \lambda)<d$ and so, since $\ell_{i,j}(\bs \lambda)\in[0,1)$, it is equivalent to $0\leq \ell_{i,j}(\bs L)<d$. Assuming that $\bs L$ satisfies this condition, $\J^*_K(\gamma, x; \boldsymbol L,\boldsymbol s)$ can be then rewritten as 
\begin{align*}
  \J^{*}_K(\gamma, x; \boldsymbol L,\boldsymbol s) =    \ssum{0 \leq r_i-L_i k/d< k/d\ \forall i,\\  r_i\equiv s_i \mod{q},\ \forall i,\\ 
    0\leq \ell_{i,j}(\br -k \bs L/d)<k/d\ \forall i,j }
  \Pi_K\big(\Phi_{\ell_{i,j}(\boldsymbol s)}(\{\ell_{i,j}(\boldsymbol r) d/k\})\big)
\end{align*}
where we used $\{\ell_{i,j}(\boldsymbol r) d/k\}=\ell_{i,j}(\boldsymbol r d/k- \bs L)$ to rewrite the summation conditions.
This sum is essentially the same sum which arises when taking $d=1$, $\gamma=(\begin{smallmatrix}0 & 1\\ -1 &0\end{smallmatrix})$, i.e. the conjugate of the sum arising in the volume conjecture. The only difference in the general case is that $\br$ is summed over a box with sides of length $k/d$ and along arithmetic progressions modulo $q$.  For fixed $d,q$ these restrictions are negligible from the analytical point of view, and the works~\cite{Ohtsuki52,OhtsukiYokota6,Ohtsuki7} can be adapted. For~$z\in(\C\smallsetminus\R)\cup (0, 1)$, define
$$ \psi_1(z):=\lf(z), \quad \psi_2(z):=\lf(1-z), \quad \psi_3(z):=-\lf(z), \quad \psi_4(z):=-\lf(1-z), $$
and let~$\bs\mu=(\mu_1,\dots,\mu_m)$ be the solution described in \cite[\S5.1]{Ohtsuki52}, \cite[\S3.3, \S4.3, \S5.3]{OhtsukiYokota6} \cite[\S3.3, \S4.3, \S5.3, \S6.3, \S7.3]{Ohtsuki7} (conjugated to agree with our definition of~$\J_K$) to the system of equations
\begin{align}\label{spe}
  &\sum_{i=1}^4\sum_{j=1}^{m_1}\kappa_{i,j}(u)\psi_i (1-\ell_{i,j}(\bs\mu)) = 0,\qquad \forall u\in\{1,\dots,m\}
\end{align}
satisfying $0<\Re(\ell_{i,j}(\bs\mu))<1$ for all $i,j$. We write
$$ \nu_i:=\exp(\mu_i), \qquad \nu_i^{1/q}:=\exp(\mu_i/q). $$
It is known~\cite{Ohtsuki52,OhtsukiYokota6,Ohtsuki7} that~$\Q(\bs \nu) = F_K$, the trace field of~$K$. It will be useful to denote
$$ F_{K, q} := \Q(\e[1/q],\bs \nu), \qquad {\tilde F}_{K, q} := \Q(\e[1/q], \bs \nu^{1/q}). $$

The following lemma will be proven in Section~\ref{prolem}.

\begin{lemma}\label{main_terms}
  Let $\bs L\in\{0,\dots, d-1\}^m$, $\bs s\in(\Z/q\Z)^m$ with $0\leq \ell_{i,j}(\bs L)<d$ for all $i,j$. Then for all $N\geq0$, we have
  \begin{align*}
    \J^{*}_K(x ; \boldsymbol L,\boldsymbol s) =\frac1{\hessK^{1/2}} \Big(\frac{2\pi ik}{ qd}\Big)^{m/2}\exp\Big(\frac{\Vol (K)-i\cs (K)}{2\pi}\frac{k}{qd}+C(\bs s)\Big)\Big(\sum_{n=0}^N\omega_{\bs s,n}\Big(\frac{2\pi i qd}{k}\Big)^n+O\Big(\frac{qd}{k}\Big)^{N+1}\Big),
  \end{align*}
  where~$0\neq \hessK \in F_{K}$, $\omega_{\bs s,0}:=1$ and
  \begin{align}\label{defcs}
    C(\bs s)&:=\sum_{i=1}^4\sum_{j=1}^{m_i}C_{i,j}(\bs s), \qquad C_{i,j}(\bs s):=\sum_{g=1}^q B_1\Big(\frac{\ang{g\bar{p}-\ell_{i,j}(\bs s)}}q\Big) \psi_i\Big(\frac{g-\ell_{i,j}(\bs\mu)}q\Big)
  \end{align}
  where $B_1$ is the $1$-st Bernoulli polynomial.
  Moreover, for all $n\geq1$, $\omega_{\bs s,n}\in {\tilde F}_{K,q}$, and for all~$\sigma\in\Gal({\tilde F}_{K,q} / F_{K,q})$, we have $\sigma(\omega_{\bs s,n})=\omega_{\bs s-\overline p \bs u,n}$ if $\sigma$ is given by
  \begin{align}\label{galau}
    \sigma(\nu_i^{1/q})=\nu_i^{1/q}\e[u_i/q] \qquad (1\leq i \leq m)
  \end{align}
  for some $u_1,\dots,u_m\in\Z$.
\end{lemma}
Applying Lemma~\ref{main_terms}, by~\eqref{splj} and~\eqref{reme}, and recalling the condition $0\leq \ell_{i,j}(\bs L)<d$ for all~$i,j$, we obtain
\begin{align*}
  \frac{ \J_{K}(\gamma x)}{ \J_{K}( x)}& =  \frac{A(p,q)^{m_1+m_4-m_2-m_3}}{\hessK^{1/2}} \Big(\frac{ k}{ d}\Big)^{\frac32} \e(-m/8) q^{-m/2}  \exp\Big(\frac{\Vol (K)-i\cs (K)}{2\pi}\frac{k}{qd}\Big)\\
  &\quad\times\sum_{\bs s \mod q}\exp(C(\bs s))\Big(\sum_{n=0}^N\omega_{\bs s,n}\Big(\frac{2\pi iqd}{k}\Big)^n+O\Big(\frac{1}{k}\Big)^{N+1}\Big),
\end{align*}
for all $N\geq0$. 

Now, $6q \deks(p,q)$ is an integer, so $A(p,q)$ is a $24q$ root of unity, and letting~$\nu_K = -m_1-m_4+m_2+m_3$, we have~$\nu_K \equiv m+1\mod{2}$. We deduce
$$ 
\e\Big(\frac{\nu_K}2 s(p, q) + \frac{\nu-m}8\Big) = \pm \e\Big(\frac{\nu_K}2s(p, q)+\frac18\Big) \omega,
 $$
where~$\omega \in \{1, i\}$ is independent of $\alpha$ and $\nu_K$ is as in the following table.
\begin{figure}[H]
  \centering
  \begin{tabular}{|c|c|c|c|c|c|c|c|c|c|c|}
    \hline
    $K$ & $4_1$ & $5_1$ & $6_1$ & $6_2$ & $6_3$ & $7_3$ & $7_4$ & $7_5$ & $7_6$ & $7_7$ \\
    \hline
    $\nu_K$ & $0$ & $1$ & $2$ & $-2$ & $0$ & $1$ & $-3$ & $-1$ & $-1$ & $-1$ \\
    \hline
  \end{tabular}
  \caption{Values of~$\nu_K$}
  \label{fig:nuK}
\end{figure}

Keeping track of the factor~$i$ implicit in~$\frac{2\pi}{\hbar} = -i\frac{k}{dq}$, the proof of Theorem~\ref{th:1} follows from the following lemma, upon possibly multiplying the value of $\hessK$ by $-1$.

\begin{lemma}\label{conste}
  There exists $0\neq U\in F_K$ such that defining
  $$S_n:= U^{1/2} \sum_{\bs s \mod q}\exp(C(\bs s))\omega_{\bs s,n},\qquad n\in\N_{\geq 0},$$
  we have 
  $S_n^q \in F_{K, q}$ for all $n\geq0$. If moreover $S_n\neq0$, then for all $n'\geq0$ we have $S_n^{-1}S_{n'}\in F_{K,q}$.
\end{lemma}

Note that in~Theorems~2.2 and~2.6 of~\cite{DimofteGaroufalidis2018}, similar computations are carried out for coefficients of power series constructed by a different process, which are conjectured to match those in the modularity conjecture.
\begin{proof}
  For $x\in(-1,1)$ we have $B_1(\ang{x})=x-\frac12+\1_{x\leq 0}$. Thus,
  \begin{align*}
    C_{i,j}(\bs s)&
    =\sum_{g=1}^{q}B_1\Big(\frac{\ang{g- \ell_{i,j}(\bs s)}}q\Big) \psi_i\Big(\frac{\ang{gp } -\ell_{i,j}(\bs\mu)}q\Big)\\
    &=\sum_{g=1}^{\ang{\ell_{i,j}(\bs s)}} \psi_i\Big(\frac{\ang{gp } -\ell_{i,j}(\bs\mu)}q\Big)+\sum_{g=1}^{q} \psi_i\Big(\frac{\ang{gp } -\ell_{i,j}(\bs\mu)}q\Big)\Big(\frac{ g}q-\frac{ \ang{\ell_{i,j}(\bs s)}}q-\frac12\Big)\\
    &=\sum_{g=1}^{\ang{\ell_{i,j}(\bs s)}} \psi_i\Big(\frac{\ang{gp } -\ell_{i,j}(\bs\mu)}q\Big)+\sum_{g=1}^{q} \psi_i\Big(\frac{\ang{gp } -\ell_{i,j}(\bs\mu)}q\Big)\frac{ g}q- \psi_i(1 -\ell_{i,j}(\bs\mu))\Big(\frac{ \ang{\ell_{i,j}(\bs s)}}q+\frac12\Big)
  \end{align*}
  where in the last equality we used~\eqref{ff4}. Then,
  \begin{align}\label{cspl}
    C_{i,j}(\bs s)&=D_{i,j}(\bs s)+\frac1q{D'_{i,j}}- \psi_i(1 -\ell_{i,j}(\bs\mu))\Big(\frac{  \ell_{i,j}(\bs s)}q+\frac12\Big),
  \end{align}
  with
  \begin{align}
    D_{i,j}(\bs s)&:=\sum_{g=1}^{\ang{\ell_{i,j}(\bs s)}} \psi_i\Big(\frac{\ang{gp } -\ell_{i,j}(\bs\mu)}q\Big) + \psi_i(1 -\ell_{i,j}(\bs\mu))\frac{\ell_{i,j}(\bs s)- \ang{\ell_{i,j}(\bs s)}}q,\label{dfnd}\\
    D'_{i,j}&:=\sum_{g=1}^{q} \psi_i\Big(\frac{\ang{gp } -\ell_{i,j}(\bs\mu)}q\Big)g.\label{dfndp}
  \end{align} 
  It follows that 
  \begin{align}\label{seq}
    \!\sum_{\bs s \mod q}\exp(C(\bs s))\omega_{\bs s,n}&=\exp\Big(\frac1q\sum_{i=1}^4\sum_{j=1}^{m_i}D'_{i,j}-\frac12E\Big)\sum_{\bs s\mod q}\exp\Big(\sum_{i=1}^4\sum_{j=1}^{m_i}D_{i,j}(\bs s)\Big)\omega_{\bs s,n}
  \end{align}
  where we used that, by~\eqref{spe},
  \begin{align*}
    E&=2 \sum_{i=1}^4\sum_{j=1}^{m_i} \psi_i(1-\ell_{i,j}(\bs\mu))\Big(\frac{  \ell_{i,j}(\bs s)}q+\frac12\Big)=\sum_{i=1}^4\sum_{j=1}^{m_i} \psi_i(1-\ell_{i,j}(\bs\mu))
  \end{align*}
  is independent of $\bs s$. In particular, writing $U=\exp(E)$, by the definition of $\psi_i$,~\eqref{dfnd}-\eqref{dfndp} and Lemma~\ref{main_terms}, we have $S_n^q\in {\tilde F}_{K, q}$.
  The extension ${\tilde F}_{K, q} / F_{K, q}$ is Galois and~$\Gal({\tilde F}_{K, q} / F_{K, q})$ consists of automorphisms of the form~\eqref{galau}. Thus, it suffices to show that $S_n^q$ is invariant under any such automorphism $\sigma$. Now, by Lemma~\ref{main_terms},
  \begin{align*}
    \sigma\big(S_n^q\big)&=\exp\Big(\frac1q\sum_{i=1}^4\sum_{j=1}^{m_i}{D'}^\sigma_{i,j}\Big)\sum_{\bs s\mod q}\exp\Big(\sum_{i=1}^4\sum_{j=1}^{m_i}D^\sigma_{i,j}(\bs s)\Big)\omega_{\bs s-\bar p \bs u,n}\\
  \end{align*}
  where
  \begin{align}
    D^\sigma_{i,j}(\bs s)&:=\sum_{g=1}^{\ang{\ell_{i,j}(\bs s)}} \psi_i\Big(\frac{\ang{gp -\ell_{i,j}(\bs u) } -\ell_{i,j}(\bs\mu)}q\Big)- \psi_i(1 -\ell_{i,j}(\bs\mu))\frac{\ell_{i,j}(\bs s)- \ang{\ell_{i,j}(\bs s)}}q,\label{defdp}\\
    (D'_{i,j})^\sigma&:=\sum_{g=1}^{q} \psi_i\Big(\frac{\ang{gp -\ell_{i,j}(\bs u) } -\ell_{i,j}(\bs\mu)}q\Big)g.\notag
  \end{align}
  The same computation as above gives
  \begin{align}\label{udav}
    &D^\sigma_{i,j}(\bs s)+\frac1q{{D'}^\sigma_{i,j}}- \psi_i(1 -\ell_{i,j}(\bs\mu))\Big(\frac{  \ell_{i,j}(\bs s)}q+\frac12\Big)=C^\sigma_{i,j}(\bs s)
  \end{align}
  where 
  \begin{align}
    &C^\sigma_{i,j}(\bs s):=\sum_{g=1}^{q}B_1\Big(\frac{\ang{g-\ell_{i,j}(\bs s)}}q\Big) \psi_i\Big(\frac{\ang{gp -\ell_{i,j}(\bs u)} -\ell_{i,j}(\bs\mu)}q\Big)
  \end{align}
  so that one finds 
  \begin{align*}
    \sigma\big(S_n^q\big)&=\exp(- q E/2)\bigg(\sum_{\bs s\mod q}\exp\bigg(\sum_{i=1}^4\sum_{j=1}^{m_i}C^\sigma_{i,j}(\bs s)\bigg)\omega_{\bs s-\bar p \bs u,n}\bigg)^q.
  \end{align*}
  By the change of variables $g\to g+\overline p\ell_{i,j}(\bs u)$ one obtains $C^\sigma_{i,j}(\bs s)=C(\bs s-\overline p\bs u)$, so that, after the change of variables $\bs s\to \bs s+\overline p \bu$, one obtains $\sigma(S_n^q)=S_n^q$, and so $S_n^q \in F_{K, q}$, as desired.

  Now, assume $S_{n'}\neq0$. By~\eqref{seq} we have
  \begin{align*}
    S_n^{-1}S_{n'}&=\bigg(\sum_{\bs s\mod q} \exp\Big(\sum_{i=1}^4\sum_{j=1}^{m_i}D_{i,j}(\bs s)\Big)\omega_{\bs s,n}\bigg)^{-1} \sum_{\bs s\mod q} \exp\Big(\sum_{i=1}^4\sum_{j=1}^{m_i}D_{i,j}(\bs s)\Big)\omega_{\bs s,n'}. 
  \end{align*}
  and so $S_n^{-1}S_{n'}\in {\tilde F}_{K, q}$. Moreover, given an automorphism $\sigma$ as in~\eqref{galau}, one shows as above that $\sigma(S_n^{-1}S_{n'})=S_n^{-1}S_{n'}$ and so $S_n^{-1}S_{n'}\in F_{K, q}$.
\end{proof}

\begin{remark}\label{c45}
The constant term in Theorem~\ref{th:1} can be worked out from the arguments above as an explicit product of algebraic numbers. In the case~$K = 4_1$, we obtain
\begin{align*}
  C_{4_1}(\alpha)D_{4_1,0}(\alpha) = {}&{} c \delta_{4_1}^{-1/2} \Lambda_{4_1, \alpha}^{1/c}, \\
  \delta_{4_1} = {}& i\sqrt{3}, \\
  \Lambda_{4_1,\alpha}^{1/c} = {}&{} \Big(\prod_{g=1}^c \abs{\omega_g}^{2g/c}\Big)\sum_{r=1}^c \prod_{g=1}^{r} \abs{\omega_g}^2,
\end{align*}
where~$\omega_g = 1-\e(g\alpha - \tfrac5{6c})$.
In the case of~$K = 5_2$, let~$\tau \approx 0.665 + 0.562i$ solve~$\tau^3-\tau+1 =0$, and let~$\mu_1 \approx 0.224 + 0.045i$ and~$\mu_2 \approx 0.164-0.067i$ be such that~$\e(\mu_1) = \tau^2$ and~$\e(\mu_2) = \tau^2+\tau$. Then
\begin{align*}
  C_{5_2}(\alpha)D_{5_2,0}(\alpha) = {}&{} \e\Big(\frac{\deks(\alpha)}2\Big) c^{1/2} \delta_{5_2}^{-1/2} \Lambda_{5_2, \alpha}^{1/c}, \\
  \delta_{5_2} = {}& 3\tau-2\tau^2, \\
  \Lambda_{5_2,\alpha}^{1/c} = {}&{} \e\Big(\mu_1\frac{c+1}{2c}\Big) \Big(\prod_{g=1}^c \omega_g^{-g/c}\chi_g^{-2g/c}\Big) \times \\
  &{} \times\sum_{r_1, r_2 =1}^c \e\Big(\frac{\mu_1(r_1+r_2) + \mu_2 r_1}c + \frac{r_1}2 - \frac{\alpha}2 r_1(1+r_1+2r_2)\Big) \Big(\prod_{g=1}^{r_1} \omega_g^{-1}\Big) \Big(\prod_{g=1}^{r_2} \chi_g^{-2}\Big),
\end{align*}
where~$\omega_g = 1-\e(-g\alpha + \tfrac{\mu_1}c)$ and~$\chi_g = 1-\e(g\alpha - \tfrac{\mu_2}c)$, and the logarithms are taken with principal determination. Note that~$F_{5_2} = \Q(\tau)$ (see~\cite{Ohtsuki52}). \end{remark}

\subsection{Proof of Lemma~\ref{extra_terms}}\label{prolex}

For $\lambda\in\R$, let~$\Lambda$ denote the Lobachevsky function
$$ \Lambda(\lambda):=-\Re\Big(\Lie(\lambda)\Big)=-\int_0^{\{\lambda\}} \log(2\sin (\pi t))\df t.$$
where the last equality follows by~\eqref{ff3} and~\eqref{dig}. Note that~$\Lambda$ is~$1$-periodic and odd.
We need to bound 
\begin{align*}
  W_K(\bs \lambda):= -\sum_{i=1}^2\sum_{j=1}^{m_i}\Lambda(\ell_{i,j}(\bs \lambda)) + \sum_{i=3}^4\sum_{j=1}^{m_i} \Lambda(\ell_{i,j}(\bs \lambda))
\end{align*}
for all~$\lambdab \in [0, 1)^m$ such that~$\ell_{i,j}(\lambdab) \not\in [0, 1)$ for some~$(i, j)$. Define
$$ M = \Lambda(1/6) \in[0.16, 0.162]. $$
We will require the following simple inequalities: for~$\alpha, \beta \in [0, 1]$,
\begin{align}
  \abs{\Lambda(\alpha)} \leq{}& M, \label{eq:ineqLambda-tot} \\
  \Lambda(\alpha) \leq {}& 0, & (\alpha \geq \tfrac12), \label{eq:ineqLambda-pos} \\
  2(\Lambda(\alpha) + \Lambda(\beta)) - \Lambda(\alpha + \beta) \leq {}& 4\Lambda(\tfrac14) < 0.59,  & (\alpha + \beta \leq 1), \label{eq:ineqLambda-sum-1} \\
  2(\Lambda(\alpha) + \Lambda(\beta)) - \Lambda(\alpha + \beta) \leq {}& M,  & (\alpha + \beta \geq 1). \label{eq:ineqLambda-sum-2} \\
  2(\Lambda(\alpha) + \Lambda(\beta)) - \Lambda(\alpha + \beta) \leq {}& 0.45,  & (\alpha + \beta \leq 1, \text{ and } \alpha \geq \tfrac12). \label{eq:ineqLambda-sum-3} \\
  \Lambda(\alpha) - \Lambda(\beta) \leq {}& 0.23, & (\alpha \leq \tfrac12\leq \beta, \text{ and } \beta \leq 2\alpha), \label{eq:ineqLambda-double}
\end{align}
The bound~\eqref{eq:ineqLambda-double} is proved by optimizing at~$\beta = \min(\frac56, 2\alpha)$. The maximum is achieved at a point~$\alpha$ where~$\rho = \sin(\pi \alpha)$ solves~$\rho(1-\rho^2) = \frac18$. Similarly, the bound~\eqref{eq:ineqLambda-sum-3} is proved by maximizing at~$\alpha$ for which~$\rho = \cos(\pi \alpha)$ solves~$\rho = 2(1-\rho^2)$.

We consider the situation case by case; in each case, we work under the extra assumption that for some~$(i, j)$, we have~$\ell_{i,j}(\lambdab) \not\in[0, 1)$.
\begin{itemize}
\item Case~$K = 5_2$. We have
  $$ W_K(\lambdab) = \Lambda(\lambda_1) +  2\Lambda(\lambda_2). $$
  A bound of~$2M \leq \frac{\Vol(K)}{2\pi} - 0.12$ is enough. Assume~$\lambda_1 + \lambda_2 \geq 1$. Then~$\lambda_i \geq \frac12$ for some~$i\in\{1, 2\}$, so that using~\eqref{eq:ineqLambda-tot} and \eqref{eq:ineqLambda-pos}, we get
  $$ W_K(\lambdab) \leq 2M $$

\item Case~$K = 6_1$. We have
  $$ W_K(\lambdab) = 2\Lambda(\lambda_1) + \Lambda(\lambda_2) + \Lambda(\lambda_3). $$
  A bound of~$3M \leq \frac{\Vol(K)}{2\pi} - 0.01$ is enough. Thus, by~\eqref{eq:ineqLambda-pos}, we may assume~$\lambda_i\leq \frac12$ for all~$i\in\{1, 2, 3\}$. Assume~$\lambda_1 + \lambda_2 + \lambda_3 \geq 1$. Then by concavity of~$\Lambda$ on~$[0, \frac12]$, we have~$W_K(\lambdab) \leq 2\Lambda(\lambda_1) + 2\Lambda(\frac12(\lambda_2 + \lambda_3))$. The bound~\eqref{eq:ineqLambda-double} can then be applied with~$(\alpha,\beta) = (\frac12(\lambda_2 + \lambda_3), 1-\lambda_1)$, and yields~$W_K(\lambdab) \leq 0.46 \leq 3M$. We find in all cases
  $$ W_K(\lambdab) \leq 3M. $$

\item Case~$K = 6_2$. We have
  $$ W_K(\lambdab) = -2\Lambda(\lambda_1) + 2\Lambda(\lambda_2) + \Lambda(\lambda_3) + \Lambda(\lambda_1 - \lambda_2). $$
  A bound of~$4M \leq \frac{\Vol(K)}{2\pi} - 0.05$ is enough. Because of~\eqref{eq:ineqLambda-pos}, we may assume~$\lambda_1>\frac12$ and~$\lambda_2 \leq \frac12$. Then the case~$\lambda_1 \leq \lambda_2$ is excluded, and we may assume~$\lambda_2 + \lambda_3 \geq 1$. Then~$\lambda_3 \geq \frac12$, and so by~\eqref{eq:ineqLambda-pos} and \eqref{eq:ineqLambda-sum-1} (with~$(\alpha, \beta) = (1-\lambda_1, \lambda_2)$), we obtain~$W_K(\lambdab) \leq 4\Lambda(\frac14)$. We obtain in all cases
  $$ W_K(\lambdab) \leq 4M. $$

\item Case~$K = 6_3$. We have
  $$ W_K(\lambdab) = -2\Lambda(\lambda_2) + 2\Lambda(\lambda_1) + 2\Lambda(\lambda_3) + \Lambda(\lambda_2-\lambda_3) + \Lambda(\lambda_2-\lambda_1). $$
  A bound of~$5M \leq \frac{\Vol(K)}{2\pi} - 0.09$ is enough.
By symmetry, we may assume~$\lambda_2 \leq \lambda_3$.  Suppose first~$\lambda_2\geq \frac12$. Then~$\lambda_2-\lambda_3 \in [-\frac12, 0]$, so by~\eqref{eq:ineqLambda-pos},~$\Lambda(\lambda_2-\lambda_3) \leq 0$ and~$\Lambda(\lambda_3)\leq 0$. We deduce~$W_K(\lambdab) \leq 5M$ by~\eqref{eq:ineqLambda-tot}. Suppose on the other hand that~$\lambda_2 < \frac12$. If~$\lambda_3\geq \frac12$, then by~\eqref{eq:ineqLambda-sum-1}-\eqref{eq:ineqLambda-sum-2} with~$(\alpha, \beta) = (\lambda_1, 1-\lambda_2)$, we find~$W_K(\lambdab) \leq 4\Lambda(\frac14) + M \leq 5M$. If, finally,~$\lambda_3 < \frac12$, then by~\eqref{eq:ineqLambda-pos} we have~$\Lambda(\lambda_2 - \lambda_3) \leq 0$, so that~$W_K(\lambdab) \leq 5M$. In all cases, we find
  $$ W_K(\lambdab) \leq 5M. $$

\item Case~$K = 7_3$. We have
  $$ W_K(\lambdab) = -2\Lambda(\lambda_2) + \Lambda(\lambda_1) + \Lambda(\lambda_3) + \Lambda(\lambda_4). $$
  A bound of~$4M \leq \frac{\Vol(K)}{2\pi} - 0.08$ is enough. Thus, by~\eqref{eq:ineqLambda-pos}, we may assume~$\lambda_j < \frac12$ for all~$j\in\{1, 3, 4\}$, and~$\lambda_2 > \frac12$. Assume that~$\lambda_2 \leq \lambda_3 + \lambda_4$. By~\eqref{eq:ineqLambda-pos} and the concavity of~$\Lambda$ on~$[0, \frac12]$, we have
  $$ W_K(\lambdab) \leq -2\Lambda(\lambda_2) + \Lambda(\lambda_3) + \Lambda(\lambda_4) + M \leq 2\Big(\Lambda\Big(\frac{\lambda_3 + \lambda_4}2\Big) - \Lambda(\lambda_2)\Big) + M. $$
  By~\eqref{eq:ineqLambda-double} with~$\beta = \frac12(\lambda_3 + \lambda_4)$, we obtain~$W_K(\lambdab) \leq 0.46 +M\leq 4M$. We deduce that in all cases,
  $$ W_K(\lambdab) \leq 4M. $$

\item Case~$K = 7_4$. We have
  $$ W_K(\lambdab) = -\Lambda(\lambda_2 + \lambda_3) + \Lambda(\lambda_1) + 2\Lambda(\lambda_2) + 2\Lambda(\lambda_3) + \Lambda(\lambda_4). $$
  A bound of~$5M \leq \frac{\Vol(K)}{2\pi} - 0.01$ is enough. By~\eqref{eq:ineqLambda-pos} and~\eqref{eq:ineqLambda-tot} we may assume without loss of generality~$\lambda_2, \lambda_3 < \frac12$. Moreover, by~\eqref{eq:ineqLambda-sum-1} (note that~$0.59 \leq 4M$), we always have~$W_K(\lambdab) \leq 4M + \Lambda(\lambda_1) + \Lambda(\lambda_4)$, and so we may assume~$\lambda_1, \lambda_4 < \frac12$ as well. But then, neither of the cases~$\lambda_i + \lambda_j \geq 1$ can occur for~$(i, j) \in \{(1, 2), (2, 3), (3, 4)\}$. We find in all cases that
  $$ W_K(\lambdab) \leq 5M. $$

\item Case~$K = 7_5$. We have
  $$ W_K(\lambdab) = -2\Lambda(\lambda_3) + 2\Lambda(\lambda_1) + \Lambda(\lambda_4) + \Lambda(\lambda_2 - \lambda_1) + \Lambda(\lambda_3 - \lambda_2). $$
  A bound of~$6M \leq \frac{\Vol(K)}{2\pi} - 0.05$ is enough. We may assume~$\lambda_3 > \frac12$ and~$\lambda_1, \lambda_4 < \frac12$. Then the case~$\lambda_3 \leq \lambda_4$ is excluded. Assume next~$\lambda_2 \leq \lambda_1$. Then~$\lambda_2 - \lambda_1 \in (-\frac12, 0]$, so that~$\Lambda(\lambda_2-\lambda_1) \leq 0$. Similarly, if~$\lambda_3 \leq \lambda_2$, then~$\lambda_2 > \frac12$, and~$\Lambda(\lambda_3 - \lambda_2) \leq 0$. In all cases, we find by~\eqref{eq:ineqLambda-tot} and \eqref{eq:ineqLambda-pos} that
  $$ W_K(\lambdab) \leq 6M. $$

\item Case~$K = 7_6$. We have
  $$ W_K(\lambdab) = -\Lambda(\lambda_3 + \lambda_4) - \Lambda(\lambda_2 + \lambda_3) + 2\Lambda(\lambda_1) + \Lambda(\lambda_2-\lambda_1) + 2\Lambda(\lambda_3) + 2\Lambda(\lambda_4). $$
  A bound of~$0.45 + 4M \leq \frac{\Vol(K)}{2\pi} - 0.03$ is enough; note that~$0.59 + 3M < 0.45 + 4M$. In particular, we may assume that~$\lambda_1 < \frac12$, since otherwise, by~\eqref{eq:ineqLambda-pos} and \eqref{eq:ineqLambda-sum-1}-\eqref{eq:ineqLambda-sum-2}, we get~$W_K(\lambdab) \leq 0.59 + 2M$. Assume first~$\lambda_3 + \lambda_4 \geq 1$. Then by~\eqref{eq:ineqLambda-sum-2}, we obtain~$W_K(\lambdab) \leq 5M$, which is acceptable. Next, assume that~$\lambda_2 + \lambda_3 \geq 1$. If~$\lambda_3 \leq \frac12$, then~$\lambda_2 + \lambda_3 \in [1, \frac32]$ and by~\eqref{eq:ineqLambda-sum-2} we have~$W_K(\lambdab) \leq 0.59 + 3M$. If on the other hand~$\lambda_3 > \frac12$, then by~\eqref{eq:ineqLambda-sum-3}, we get~$W_K(\lambdab) \leq 0.45 + 4M$. Both bounds are acceptable. Finally, assume that~$\lambda_2 \leq \lambda_1$. Then~$\lambda_2 - \lambda_1 \in (-\frac12, 0]$, and so~$\Lambda(\lambda_2-\lambda_1) \leq 0$, and we again obtain~$W_K(\lambdab) \leq 0.59 + 3M$. We find in all cases that
  $$ W_K(\lambdab) \leq 0.45 + 4M. $$

\item Case~$K = 7_7$. We have
  $$ W_K(\lambdab) = 2(\Lambda(\lambda_1)+\Lambda(\lambda_2)+\Lambda(\lambda_3)+\Lambda(\lambda_4)) - \Lambda(\lambda_1 + \lambda_2) - \Lambda(\lambda_2+\lambda_3) - \Lambda(\lambda_3+\lambda_4). $$
  Assume~$\lambda_i + \lambda_j \geq 1$ for some~$(i, j)\in \{(1, 2), (2, 3), (3, 4)\}$. Then by~\eqref{eq:ineqLambda-sum-2} and~\eqref{eq:ineqLambda-tot}, we have
  $$ W_K(\lambdab) \leq 7M \leq \frac{\Vol(K)}{2\pi} - 0.08. $$
\end{itemize}

Summarizing the above, we find that in all cases considered for~$K$, there holds
\begin{equation}
  \label{eq:values-EK}
  E_K:=\sup\{W_K(\bs \lambda)\mid \bs \lambda\in[0,1)^n,\ \exists i,j,\ \ell_{i,j}(\bs \lambda)\notin[0,1)\} \leq \frac{\Vol K}{2\pi}-0.01.
\end{equation}
This proves Lemma~\ref{extra_terms}.

\begin{remark}
  The analogue of Lemma~\ref{extra_terms} for the knot~$7_2$ is false as stated. It is likely that this obstacle can be lifted by processing the contour integral arguments underlying Lemma~\ref{main_terms} more carefully (see~\cite[Remark~8.1]{Ohtsuki7}). For sake of clarity, and since our main point is rather to stress how the modularity conjecture can be reduced to the arithmeticity conjecture, we chose to omit the case~$K = 7_2$.
\end{remark}

\subsection{Proof of Lemma~\ref{main_terms}}\label{prolem}

\begin{remark}\label{conjex}
  We have $\bar{\Lie(\lambda)}=-\Lie(1-\lambda)$ for $\lambda\in[0,1)$ and $\overline {\lf(z)}=\lf(1-\overline z)$ for $z\in (\C \setminus \R)\cup(0,1)$. In particular we can write the conjugates of $\Phi_{r}(\lambda)$ and $g_d(\lambda)$ given in~\eqref{defb} and~\eqref{defg2} (\eqref{defg} if $\lambda=0$) as
  \begin{align*} 
    \overline{\Phi_{r}(\lambda)} &= \exp\bigg(-\frac k{qd} \Big(\Lie(1-\lambda) +\frac{\pi i }{12}\Big)-\frac{\pi i d}{12kq}+ {\Er_r^*(\lambda,d/k)} \bigg)\\
    {  \Er_r^*(\lambda,d/k)}&:= \frac12 f\Big(\frac{q-\ang{pr}+\lambda}q\Big)+ \sum_{g=1}^q i\int_{0}^\infty   \bigg( \frac{\lf(\frac{q-g+\lambda}q - i\frac{dt}k)}{\e[ -\frac{g\bar{p}-r}q]\e^{2\pi t}-1}-   \frac{\lf(\frac{q-g+\lambda}q + i\frac{dt}k)}{\e[ \frac{g\bar{p}-r}q]\e^{2\pi t}-1}\bigg)\df{t},\quad (\lambda\neq0)
  \end{align*}
  and $ \Er_r^*(0,d/k)=\overline { \Er_r(0,d/k)}$. In particular, we can extend $\overline{\Phi_{r}(\lambda)} $ to a holomorphic function of~$\lambda$ in the strip $0<\Re(\lambda)<1$.
\end{remark}

In the following lemmas we give some properties of the expansion of $\Er_r$ and $\Er_r^*$.

\begin{lemma}\label{expg0}
  Assume $\Re(\lambda)\in(\eps,1-\eps)$ with  $\eps\in(0,1/2)$ and let $s\in\Z/q\Z$.  Then, for all $M\geq0$ uniformly in $\lambda$ we have
  \begin{align} \label{gexp}
    \Er_s(\lambda,d/k)=\sum_{\ell=0}^M\Big(\frac {qd}k\Big)^\ell \Er_{s,\ell}(\lambda) +O_\eps\Big(\frac {qd}k\Big)^{M+1},\quad
    \Er_s^*(\lambda,d/k)=\sum_{\ell=0}^M\Big(\frac {qd}k\Big)^\ell \Er_{s,\ell}^*(\lambda) +O_\eps\Big(\frac {qd}k\Big)^{M+1},
  \end{align}
  where
  \begingroup
  \allowdisplaybreaks
  \begin{align*} 
    \Er_{0,\ell}(\lambda)&:=
    \sum_{g=1}^q f\Big(\frac{g-\lambda}q \Big)  B_{1}\Big( \frac{\ang{g\bar{p}-s} }q\Big),\qquad
    \Er^*_{0,\ell}(\lambda):=
    \sum_{g=1}^q f\Big(\frac{g-\lambda}q \Big)  B_{1}\Big( \frac{\ang{g\bar{p}-s} }q\Big),\\
    \Er_{s,\ell}(\lambda)&:=  \frac{(-1)^\ell}{q^\ell (\ell+1)!} \sum_{g=1}^q\lf^{(\ell)}\Big(\frac{g-\lambda}q \Big)   \tilde B_{\ell+1}\Big( \frac{g\bar{p}-s}q\Big),\qquad \ell\geq1,\\
    \Er^*_{s,\ell}(\lambda)&:=  \frac{(-1)^\ell}{q^\ell (\ell+1)!} \sum_{g=1}^q\lf^{(\ell)}\Big(\frac{1-g+\lambda}q \Big)   \tilde B_{\ell+1}\Big( \frac{g\bar{p}-s}q\Big),\qquad \ell\geq1.
  \end{align*}
  \endgroup
\end{lemma}
\begin{proof}
  For $\Re(\lambda)\in(\eps,1-\eps)$ we have~$\inf_{g\in\Z}\|\frac{g-\lambda}q\| \gg_\ee 1$. Thus, since
  \begin{align*}
    \lf'(z) = \pi (\cot(\pi z)+i),\qquad \lf^{(\nu)}(z) = (-1)^{\nu-1}(\nu-1)! \sum_{n\in\Z} \frac{1}{(z-n)^\nu},\qquad (z\notin\Z,\ \nu\geq2), 
  \end{align*}
  we have that $\lf^{(\nu)}(\frac{g-\lambda+it}q)\ll_{q,\nu,\eps} (1+|t|^\eps)$ for all $\nu\geq0$, $t\in\R$. The Lemma then follows immediately by expanding in Taylor series and applying Lemma~\ref{lem:B-integral}.
\end{proof}

\begin{lemma}\label{expg}
  Assume $\Re(\tilde \lambda)\in(0,1)$ and let $s\in\Z/q\Z$. Let $h_{s,\ell}$ be either $\Er_{s,\ell}$ or $\Er^*_{s,\ell}$. Then, for all $\ell,M\geq0$, and all~$\lambda$ in a neighborhood of~$\tilde\lambda$, we have
  \begin{align}\label{exp1}
    h_{s,\ell}(\lambda)&= \sum_{v=0}^M (2\pi i)^{\ell+v} C_{s,\ell,v}(\tilde \lambda)(\lambda-\tilde \lambda)^v+O(|\lambda-\tilde \lambda|^{M+1}),
  \end{align}
  where, for $(v,\ell)=(0,0)$,
  \begin{align}
    C_{0,0}(\tilde \lambda)&= \sum_{g=1}^q f\Big(\frac{g-\tilde\lambda}q \Big)  B_{1}\Big( \frac{\ang{g\bar{p}-s} }q\Big),&& \text{if }h_{s,\ell}=\Er_{s,\ell}\label{exp01}\\
    C_{0,0}(\tilde \lambda)&= \sum_{g=1}^q f\Big(\frac{q-g+\tilde\lambda}q \Big)  B_{1}\Big( \frac{\ang{g\bar{p}-s} }q\Big),&& \text{if }h_{s,\ell}=\Er^*_{s,\ell}.\label{exp02}
  \end{align}
  Moreover, for $(v,\ell)\neq(0,0)$ we have  $C_{s,\ell,v}(\tilde \lambda)\in\Q(\e[\tfrac{1}{q}],\e[\tfrac{\tilde\lambda}{q}])$; also if 
  $$\sigma\in \textnormal{Gal}\big(\Q(\e(\tfrac{1}{q}),\e(\tfrac{\tilde\lambda}{q}))/\Q(\e(\tfrac{1}{q}),\e(\tilde\lambda))\big)$$ 
  is such that $\sigma(\e[\tilde\lambda/q])=\e[\tilde(\lambda+u)/q]$ for some $u\in\{0,\dots,q-1\}$, then
  \begin{align}\label{pdc}
    \sigma(C_{s,\ell,v}(\tilde \lambda))=C_{s-u\overline p,\ell,v}(\tilde \lambda).
  \end{align}
\end{lemma}
\begin{proof}
  The equations~\eqref{exp1}-\eqref{exp02} follow immediately by Taylor expansion. Moreover, if $\ell\geq1$, $v\geq0$ and $h_{r,\ell}=\Er_{r,\ell}$ then
  \begin{align*}
    C_{r,\ell,v}(\tilde \lambda)=  \frac{(-1)^\ell}{q^\ell (\ell+1)!}  \frac{1}{q^v v!}\sum_{g=1}^q (2 \pi i)^{\ell+v}\lf^{(\ell+v)}\Big(\frac{g-\tilde \lambda}q \Big)   \tilde B_{\ell+1}\Big( \frac{g\bar{p}-r}q\Big),
  \end{align*}
  so that $C_{r,\ell,v}(\tilde \lambda)\in \Q(\e[\tfrac{1}{q}],\e[\tfrac{\tilde\lambda}{q}])$ since $\lf'(z) = \pi (\cot(\pi z)+i)= \frac{2 i}{1-\e[-z]}$ and~\eqref{pdc} follows by the change of variables $g\to g+j$.
  The case $\ell=0$, $v\geq1$ and the analogous property for $\tilde \Er_{r,\ell}$ can be proven in the same way.
\end{proof}

\begin{proof}[Proof of Lemma~\ref{main_terms}]
  As mentioned in the introduction, for the knots under consideration, the asymptotic expansion stated in Lemma~\ref{main_terms} will be essentially reduced to a proof of the asymptotic expansion in the volume conjecture for those knots. Thus, we shall frequently refer to~\cite{Ohtsuki52,OhtsukiYokota6,Ohtsuki7} where this asymptotic expansion was proven for hyperbolic knots with $5,6$ and $7$ crossings. The case of the knot $4_1$ is easier, since there is a dominant critical point on $(0,1)$, and the method of stationary phase can be applied, similarly as in the proof of Theorem~\ref{th:2} below. Thus, we will focus here on the case $K\neq 4_1$. Recall also that we assume~$K\neq 7_2$.

  By Remark~\ref{conjex}, for $0\leq \ell_{i,j}(\bs L)<d$ we can write $J^{*}_K$ as
  \begin{align*}
    \J^{*}_K(x ; \boldsymbol L,\boldsymbol s) =   \ssum{ \br d/k-\bs L\in \mathcal  D,\\  r_i\equiv s_i \mod{q},\ \forall i}
    \exp\Big(\frac{k}{qd}  V_{\bs s,d/k}(d\br/k -\bs L )\Big),
  \end{align*}
  where $\mathcal D=\{\bs n\in [0,1)^m\mid \ell_{i,j}(\bs n)\in[0,1)\, \forall i,j\}$,
  \begin{align*}
    V_{\bs s,\kappa}(\bs n):=\hat V(\bs n)+q\kappa \Big(U_{\bs s,\kappa}(\bs n)+\frac{ \pi i(m_1+m_4-m_2-m_3)}{12q}\kappa\Big)
  \end{align*}
  and
  \begin{align*}
    \hat V(\bs n)&=\sum_{j=1}^{m_1}\Big(\Lie(\ell_{1,j}(\bs n))+\frac{\pi i}{12}\Big) -\sum_{j=1}^{m_2}\Big(\Lie(1-\ell_{2,j}(\bs n))+\frac{\pi i}{12}\Big) \\
    &\quad-\sum_{j=1}^{m_3}\Big(\Lie(\ell_{3,j}(\bs n))+\frac{\pi i}{12}\Big)+\sum_{j=1}^{m_4}\Big(\Lie(1-\ell_{4,j}(\bs n))+\frac{\pi i}{12}\Big),\\
    U_{\bs s,\kappa}(\bs n)&= \sum_{j=1}^{m_1}\Er_{\bs s}(\ell_{1,j}(\bs n),\kappa)+\sum_{j=1}^{m_2}\Er_{\bs s}^*(\ell_{2,j}(\bs n),\kappa)-\sum_{j=1}^{m_3}\Er_{\bs s}(\ell_{3,j}(\bs n),\kappa)-\sum_{j=1}^{m_4}\Er_{\bs s}^*(\ell_{4,j}(\bs n),\kappa).
  \end{align*}
  The function $\hat V(\bs n)$ coincides, up to conjugation, with the limiting value of the potential function of the hyperbolic structure of the knot complement given in~\cite[(10)]{Ohtsuki52},~\cite[p.297, p.308, p.322]{OhtsukiYokota6} and~\cite[p.12, p.26, p.38, p.50, p.63]{Ohtsuki7}. We remark that the expressions for $\hat V$ given there differ from the one we have here, however it is easy to see that the two expressions actually coincide upon using the dilogarithm identity (or the formula \eqref{ff1})
  $$
  \Lie(1-\lambda)+\Lie(\lambda)=-\pi i B_2(\lambda).
  $$
  In~\cite[Lemma 2.1]{Ohtsuki52},~\cite[\S3.2, \S4.2, \S5.2]{OhtsukiYokota6} and~\cite[\S3.2, \S4.2, \S5.2, \S6.2, \S7.2]{Ohtsuki7} it was shown that for all knots under consideration $\Re (\hat V)$ is smaller than $\frac{\Vol(K)}{2\pi}$ on the boundary of $\mathcal D$. More precisely, there exists a domain $\mathcal D'\subset\mathcal D$ with $\tn{dist}(\mathcal D',\partial \mathcal D)>0$ such that $\Re(\hat V(\bs n))<\frac{\Vol(K)}{2\pi}-\delta'$ for all $\bs n\in \mathcal{D}\setminus\mathcal D'$ and some $\delta'>0$. Thus, by~\eqref{bdfg}, we have
  \begin{align*}
    \J^{*}_K(\gamma, x;\boldsymbol L,\boldsymbol s) =   \ssum{ \br d/k-\bs L\in \mathcal  D',\\  r_i\equiv s_i \mod{q},\ \forall i}
    \exp\Big(\frac{k}{qd}  V_{\bs s,d/k}(d\br/k -\bs L )\Big)+O\bigg(k^m\exp\Big(\Big(\frac{\Vol(K)}{2\pi}-\delta'\Big)\frac{k}{qd}\Big)\bigg).
  \end{align*}
  We now apply Poisson summation formula in the form of~\cite[Proposition 4.6]{Ohtsuki52} (with $k/d$ playing the role of $N$ of~\cite{Ohtsuki52}). Note that our sum are restricted to arithmetic progressions modulo $q$; since $q$ is fixed, this does not affect the argument.
  By~\cite[Lemma 5.1]{Ohtsuki52},~\cite[Lemma 3.4, 4.3 and 5.2]{OhtsukiYokota6} and~\cite[Lemma 3.2, 4.2, 5.2, 6.2 and 7.2]{Ohtsuki7} we have that
  $\hat V(\bs n)-\frac{\Vol(K)}{2\pi }$ satisfies the conditions (41)-(42) of~\cite[Proposition 4.6]{Ohtsuki52} and by~\cite[Remark 4.8]{Ohtsuki52} and Lemma~\ref{expg0} we can apply Proposition 4.6 of \cite{Ohtsuki52} to $ V_{\bs s,\bs L,d/k}(\bs n)$ rather than $\hat V(\bs n)$. We find
  \begin{align}\label{intf}
    \J^{*}_K(\gamma, x; \boldsymbol L,\boldsymbol s) = \Big(\frac k{qd}\Big)^m   \iint_{\mathcal D'}
    \exp\Big(\frac{k}{qd}  V_{\bs s,d/k}(\bs z)\Big)\df \bs z+O\bigg(\exp\Big(\Big(\frac{\Vol(K)}{2\pi}-\delta''\Big)\frac{k}{qd}\Big)\bigg),
  \end{align}
  for some $\delta''>0$, where the extra factor $q^{-m}$ comes from the restriction to the congruence classes.
  One can then apply the saddle point method in the form of~\cite[Proposition~3.5 and Remark~3.6]{Ohtsuki52} as done in~\cite[p.705-706 and \S5.2]{Ohtsuki52} (cf. also~\cite{Yokota2011}),~\cite[p.297 and \S3.5; p.309 and \S4.5; p.47 and \S5.5]{OhtsukiYokota6} and~\cite[p.13 and \S3.5; p.26 and \S4.5; p.39 and \S5.5; p.50 and \S6.5; p.64 and \S7.5]{Ohtsuki7}. Notice that both $ V_{\bs s,d/k}(\bs z)$ and the corresponding functions studied in these papers converge uniformly to $\hat V$, so the same computations apply. We then find that for all $N\geq1$ the first summand on the right of~\eqref{intf} is equal to
  \begin{align*}
    \Big(\frac{2\pi k}{ qd}\Big)^{m/2}\frac1{\det(-\Hess)^{1/2}}
    \exp\Big(\frac{k}{qd} \hat  V(\bs \mu)+C(\bs s)\Big)\bigg(1+\sum_{n=1}^N\omega_{\bs s,n}\big(\frac{2\pi i qd}{k}\Big)^n+O\Big(\frac{1}{k^{N+1}}\Big)\bigg)
  \end{align*}
  where $\bs\mu$ is a critical point of $\hat V(\bs z)$ (and thus satisfies~\eqref{spe}) with $0<\Re(\ell_{i,j}(\bs\mu))<1$ for all~$i,j$, such that $\hat V(\bs \mu)=\frac{\Vol(K)-i\cs(K)}{2\pi }$  
  and $\Hess$ is the Hesse matrix of $\hat V$ at $\bm\mu$. In particular, $\det(-\Hess)\in (\pi i)^m F_{K}$. By Lemma~\ref{expg} the coefficients $\omega_{\bs s,n}$ are in ${\tilde F}_{K, q}$ for all $n\geq1$ and if $\sigma$ is as in~\eqref{galau}, then $\sigma(\omega_{\bs s,n})=\omega_{\bs s-\overline p\bs u,n}$. Finally, by~\eqref{exp01}-\eqref{exp02} we have that $C(\bs s)$ is as in~\eqref{defcs}.
\end{proof}

\section{Proof of Theorems~\ref{th:2}}\label{pmt2}
We proceed in a similar way as in the proof of Theorem~\ref{th:1}. For $x=\bar h/k$, with $(h,k)=1$ and $1\leq h<k$, we write $ \J_{4_1}( x)$ as in~\eqref{exja} and we divide the sum into congruence classes modulo $h$, that is we write
\begin{align}\label{splj2}
  \J_{4_1}(\bar h/k) = k^{\frac{3-m}2} \sum_{0\leq s\leq h-1}\J_{4_1}(\bar h/k; s)
\end{align}
where, for $0\leq s\leq h-1$, we write
\begin{align}\label{splik2}
  \J_{4_1}(\bar h/k; s) =   \ssum{0\leq r<k,\\ r\equiv s \mod{h}}{| [ \bar h/k]_{r}|^2}.
\end{align}
with $[\cdot]_r$ as in~\eqref{sb}. We apply Theorem~\ref{thp}, which for $0\leq s<h$, $s\equiv r\mod h$, gives
\begin{align}\label{arf2}
  | [\overline{h}/k]_{r}|^2 = | [\bar k/h]_{s}|^2 \Phi_s^\dagger(r /k)\exp\Big(- \frac {2\pi r}{hk}c_0({\overline k}/h)+\Er^\dagger_s(r/k)\Big)
\end{align}
for some $\Er^\dagger_s(r/k)$ satisfying $|\Er^\dagger_s(r/k)|\ll E(h,k)$ for all $0\leq r<k$, where
\begin{align*}
  E(h,k):= 1+ \log \frac kh+\frac{k}{h^2}+\max_{r'=0,\dots, h-1}\Big|\ssum{1\leq n\leq r'}\cot\Big(\pi\frac{n\overline k }h\Big)\frac{ n}{hk}\Big|, \end{align*}
(note that if~$0\leq r_0 <h$ with~$r_0\equiv r \mod{q}$, then~$\frac{r}k \in[\frac{r_0}k, 1-\frac hk(1-\{\frac{r_0-k}h\})]$), and where for $\lambda\in[0,1]$, we define
\begin{align*} 
  \Phi_s^\dagger(\lambda) &= \exp\bigg(2\frac k{h}\Re \Big(\Lie(\lambda) \Big)\bigg)=\exp\bigg(\frac k{h}\int_0^\lambda\log(4\sin(\pi t)^2)\,dt\bigg),
\end{align*}
by~\eqref{ff3} and~\eqref{dig}.
By positivity, it follows that
\begin{align}\label{erfac}
  \J_{4_1}(\bar h/k; s)= \exp\big(O(E(h,k)+|c_0({\overline k}/h)|/h)\big)|[ \bar k/h]_{ s}|^2
  \ssum{0\leq r<k\\ r\equiv s \mod{h}}{ \Phi_s^\dagger(r/k)}.
\end{align}
Now, the function $\lambda\to \int_0^\lambda\log(4\sin(\pi t)^2)\,dt$ is continuous on $[0,1]$ and it has a unique maximum in this interval, located at $\lambda=5/6$. Moreover, it can be expanded in a neighborhood of this point as
\begin{align}\label{apin}
  \int_0^\lambda\log(4\sin(\pi t)^2)\df t=\frac{\Vol (4_1)}{2\pi}-\pi\sqrt 3(\lambda-\tfrac56)^2+O(|\lambda-\tfrac56|^3).
\end{align}
since $\int_0^{5/6}\log(4\sin(\pi t)^2)\df t=\frac{\Vol (4_1)}{2\pi}$.
It follows that 
\begin{align}\label{eastp}
  \ssum{0\leq r<k\\ r\equiv s \mod{h}} \Phi_s^\dagger(r/k)=\sqrt{\frac{k}{h\sqrt 3}}\exp\Big(\frac{\Vol (4_1)}{2\pi}\frac kh\Big)(1+O(h/k))=\exp\Big(\frac{\Vol (4_1)}{2\pi}\frac kh+O(\log (1+k/h))\Big),
\end{align}
where the first step is justified in a standard way~\cite[p.517]{FlajoletSedgewick2009}, e.g. by smoothly restricting the sum to the terms where $r/k$ is in a neighborhood of $5/6$, using~\eqref{apin} and applying the Poisson summation formula. The second step instead follows immediately by positivity.

By~\eqref{splj2},~\eqref{erfac} and~\eqref{eastp} we then find
\begin{align}\label{exts}
  \J_{4_1}(\bar h/k)=  \J_{4_1}(\bar k/h)\exp(O(E(h,k)+|c_0({\overline k}/h)|/h)),
\end{align}
as desired.

With a similar argument we can prove the following theorem, which deals with the case where the dominating term on the right hand side of~\eqref{arf2} is $\exp(- \frac {2\pi r}{hk}c_0({\overline k}/h))$.

\begin{theo}\label{th:4}
  Let $1\leq h\leq k$ with $(h,k)=1$ and assume $\cz({\overline k}/h)<0$. Then,
  \begin{equation}\label{eq:srf}
    \begin{split}
      \log |\J_{4_1,0}\big(\e[{\overline{h}}/k]\big)|&=\log |\J_{4_1,0}\big(\e[{\overline k}/h]\big)|- \frac{2\pi}{h} \cz\big({\overline k}/h\big)+O\Big(\frac k h+\max_{0\leq r' < h}\Big|\ssum{1\leq n\leq r'}\cot\Big(\pi\frac{n\overline k }h\Big)\frac{ n}{hk}\Big|\Big).
    \end{split}
  \end{equation}
\end{theo}

To prove~\eqref{eq:srf} first we observe we can assume $k>2h$ since otherwise the result is trivial. Also, we observe that, bounding trivially $\Phi_s^\dagger$, one can write~\eqref{arf2} as 
\begin{align*}
  | [\overline{h}/k]_{r}|^2 = | [\bar k/h]_{s}|^2 \exp\Big(- \frac {2\pi r}{hk}c_0({\overline k}/h)+O(E(h,k)+k/h)\Big)
\end{align*}
and so
\begin{align*}
  \J_{4_1}(\bar h/k; s)= \exp\big(O(E(h,k)+k/h)\big) |[ \bar k/h]_{ s}|^2
  \ssum{0\leq r<k\\ r\equiv s \mod{h}}\exp\Big(- \frac {2\pi r}{hk}c_0({\overline k}/h)\Big).
\end{align*}
Now, if $x\geq0$, we have $\sum_{0\leq n\leq m}\e^{nx}=\e^{mx+O(m)}$, and thus if $x$ is large the sum is roughly dominated by the last term. Then for $c_0({\overline k}/h)<0$, we have 
\begin{align*}
  \ssum{0\leq r<k\\ r\equiv s \mod{h}}\exp\Big(- \frac {2\pi r}{hk}c_0({\overline k}/h)\Big)=\exp\Big(- \frac {2\pi r'}{hk}c_0({\overline k}/h)+O(k/h)\Big)
\end{align*}
where $r'$ is the maximum integer satisfying $0\leq r'< k$ with $r'\equiv s\mod h$. Then, $k-h\leq r'<k$ and so in particular 
$\frac{r'}{hk}c_0({\overline k}/h)=\frac{1}{h}c_0({\overline k}/h)+O(\frac{1}{k}|c_0({\overline k}/h)|)$ and the result follows.

\section{Proof of Theorem~\ref{theo:cfep}}\label{fise}

Before starting, we state some basic properties of continued fractions (see~\cite{Kin} for a reference).  Given $h/k\in\Q\cap(0,1)$ with $h,k\in\N$, $(h,k)=1$, we denote by $[0;b_1,\dots,b_r]$ the continued fraction expansion of $h/k$. Then for $0\leq s\leq r$ the  convergents of $h/k$ are the fractions $[0;b_1,\dots,b_s]=\frac{u_s}{v_s}$ with $(u_s,v_s)=1$ (as usual $u_0/v_0=0/1$, $v_{-1}:=0$, $u_{-1}:=1$); the $v_s$ are called the partial quotients. The partial quotients satisfy the bounds $v_s\ll 2^{-s/2}$, $v_{r-s}\ll  k2^{-s/2}$ 
for $0\leq s\leq r$, and  $v_{s}/v_{s-1}\leq b_{s}+1$ for $1\leq s\leq r$. Also, $r\ll \log k$. For all $1\leq s\leq r$ we have $v_su_{s-1}-v_{s-1}u_s=(-1)^{s}$ and so $\frac{\overline{v_{s-1}}}{v_s}\equiv (-1)^{s+1}\frac{u_s}{v_s}\mod 1$.
Moreover, if $1\leq h'\leq k$ is such that $h'\equiv (-1)^{r+1}\overline h\mod k$, then the Euclid algorithm on $h'$ and $k$ can be written as
\begin{equation}\label{eq:aaa}
  \begin{aligned}
    &v_r=k,\qquad v_{r-1}= h',\\
    &v_{\ell+1}=b_{\ell+1}v_{\ell}+v_{\ell-1},\qquad \ell=0,\dots, r-1.\\
  \end{aligned}
\end{equation}

The following technical result, proved in~\cite{BettinDrappeaub}, will be needed in the proof of Theorem~\ref{theo:cfep}.

\begin{lemma}[{\cite[Theorem~1]{BettinDrappeaub}}]\label{parcot}
  Let $1\leq h<k$ with $(h,k)=1$. Let $v_0,\dots,v_r$ be the partial quotients of $h/k$. Then uniformly for $1\leq c\leq k$, we have
  \begin{align*}
    &\ssum{1\leq n\leq c}\cot\Big(\pi\frac{n h }k\Big)\frac{ n}{k}\ll \sum_{m=0}^{r-1} v_m\log(v_{m+1}/v_{m})+k.
  \end{align*}
\end{lemma}

\begin{proof}[Proof of Theorem~\ref{theo:cfep}]
  Since $J$ is even, applying repeatedly the reciprocity formula~\eqref{eq:frf} and the recurrence relation~\eqref{eq:aaa}, we see that for all $0\leq s\leq r$
  \begin{align*}
    \log \J_{4_1,0}(\e[ { h}/k]) = {}& \log \J_{4_1,0}(\e[ \overline{ v_{r-1}}/v_r]) = C\sum_{\ell=s+1}^{r}\bigg(\frac{v_{\ell}}{v_{\ell-1}}+E({v_{\ell-1}}/{v_{\ell}})\bigg)+    \log \J_{4_1,0}(\e[ \overline{ v_{s-1}}/v_{s}]),
  \end{align*}
  with $C=\frac{\Vol(4_1)}{2\pi}$ and $E$ satisfying~\eqref{eq:ber}.    Now, by Lemma~\ref{parcot}
  \begin{align*}
    \sum_{\ell=1}^{r}\max_{r'=0,\dots,v_{\ell-1}-1}\Big|\ssum{1\leq n\leq r'}\cot\Big(\pi\frac{n\overline v_{\ell-2} }{v_{\ell-1}}\Big)\frac{ n}{v_{\ell-1}^2}\Big|\ll
    \sum_{\ell=1}^{r} \sum_{m=0}^{\ell-2} \frac{v_m}{v_{\ell-1}}\log(v_{m+1}/v_{m})+\sum_{\ell=1}^{r}1/v_{\ell-1},
  \end{align*}
  since $\tfrac{\overline v_{\ell-2} }{v_{\ell-1}}=(-1)^{\ell}\tfrac{u_{\ell-1} }{v_{\ell-1}}\mod 1$ and so its partial quotients are $v_0\dots,v_{\ell-1}$. The second sum is $O(1)$, whereas changing the order of summation and using $v_{\ell-1-n}\ll  v_{\ell-1}2^{-n/2}$, for $n\geq0$,
  we obtain
  \begin{align*}
    \sum_{\ell=1}^{r} \sum_{m=0}^{\ell-2} \frac{v_m}{v_{\ell-1}}\log(v_{m+1}/v_{m})\ll  \sum_{m=0}^{r-2} \sum_{\ell=m+2}^{r}  2^{(m-\ell)/2}\log(v_{m+1}/v_{m})\ll\sum_{m=0}^{r-2} \log(v_{m+1}/v_{m}).
  \end{align*}
  Now, we fix an $\eps>1/k$ and we take $s$ to be the least integer in $\{1,\dots,r\}$ such that  $v_{s}\geq1/\eps$ and notice that, since $v_m\gg 2^{m/2}$, we have $s=O_\eps(1)$. Then, the above computations and~\eqref{eq:ber} give
  \begin{align*}
    \log \J_{4_1,0}(\e[ { h}/k])= \sum_{\ell=s+1}^{r}\frac{v_{\ell}}{v_{\ell-1}}(C+O(\sqrt{\eps}))+\sum_{\ell=1}^{r} O(\log(v_{\ell}/v_{\ell-1}))+    \log \J_{4_1,0}(\e[\pm{ u_{s}}/v_{s}]).
  \end{align*}
  since $\tfrac{\overline v_{s-1} }{v_{s}}=\pm\tfrac{u_{s} }{v_{s}}\mod 1$ for $\pm1:=(-1)^{s-1} $. Now, if $s\geq1$, we have $\pm u_s/v_s=\gamma (b_s)$, where 
  $\gamma:=(\begin{smallmatrix}\pm u_{s-1} & u_{s-2}\\ \pm v_{s-1} & v_{s-2}\end{smallmatrix})\in\SL(2,\Z)$. Notice that by definition of $s$ all entries of $\gamma$ are bounded by $1/\eps$.
  Thus, by Theorem~\ref{th:1}, we have
  \begin{align*}
    \log \J_{4_1,0}(\e[\pm{ u_{s}}/v_{s}]) = Cb_s+O(\log b_s)+O_\eps(1)=C\sum_{\ell=1}^s b_\ell+O(\log b_s)+O_\eps(1),
  \end{align*}
  since $s,v_{\ell-1}=O_\eps(1)$ and so also $b_\ell=O_{\eps}(1)$ for all $\ell\leq s-1$. Then, since $v_{\ell}/v_{\ell-1}=b_\ell+O(1)$ we find
  \begin{align*}
    \log \J_{4_1,0}(\e[ { h}/k])=(C+O(\sqrt\eps)) \sum_{\ell=1}^{r}(b_\ell+ O(\log b_\ell))+O(r)+O_\eps(1).
  \end{align*}
  Finally, we observe that $b_\ell+ O(\log b_\ell)=b_\ell(1+O(\sqrt\eps))+O(1/\eps)$. Thus,
  \begin{align*}
    \log \J_{4_1,0}(\e[ { h}/k])=(C+O(\sqrt\eps)) \sum_{\ell=1}^{r}b_\ell+O_\eps(r)= \Sm(h/k) \bigg(C +O(\sqrt\eps)+O_\eps\bigg(\frac{r(h/k)}{\Sm(h/k)}\bigg)\bigg).
  \end{align*}
  By hypothesis $r(h/k)/\Sm(h/k)\to0$ and so~\eqref{eq:cfef} follows by letting $\eps\to0^+$ sufficiently slowly.
Equation~\eqref{eq:asy} then follows immediately from~\cite[Corollary~1.4]{BettinDrappeau}.
\end{proof}

\begin{proof}[Proof of Theorem~\ref{th:pz}]
  We extend $\gh_{K}$ to a function on $\R_{>0}$ by setting $\gh_{K}(x):=\lim_{y\to x, y\in\Q}\gh_{K}(y)$ for all $x\not\in\R_{>0}\setminus \Q$. 
  By hypothesis $\gh_{K}(x)$ is well defined and it is easy to prove that $\gh_{K}$ is continuous on $\R\setminus\Q$. 
  Then, setting $\psi(x):=\gh_{K}(x)-\frac{\tn{Vol} (K)}{2\pi}\frac{1}x-\frac{3}2\log (1/x)$, for $x>0$, and $\psi(0)=0$, we have that $\psi(x)$ is bounded and continuous almost everywhere on $[0,1]$. Thus, by Lebesgue's integrability condition, $\psi$ is Riemann-integrable on $[0,1]$. In particular, for all $\eps>0$  there exist a differentiable function $\psi_\eps:[0,1]\to\R$ such that $\|\psi-\psi_\eps\|_{\infty,[0,1]}\leq\eps$.

  By definition $\log \J_{K,0}(\e[h/k])-\log \J_{4_1,0}(\e[k/h])=\psi(h/k)+\frac{\tn{Vol} (K)}{2\pi}\frac{k}h+\frac{3}2\log (k/h)$. Thus, proceeding as in the proof of Theorem~\ref{theo:cfep} using this formula instead of Theorem~\ref{th:2}, we obtain
  \begin{align*}
    \log \J_{4_1,0}(\e[ {\overline h}/k]) 
    &    = \sum_{\ell=1}^{r}\bigg(\frac{\tn{Vol} (K)}{2\pi}\frac{v_{\ell}}{v_{\ell-1}}+\frac{3}2\log ({v_{\ell}}/{v_{\ell-1}}) +
    \psi({v_{\ell-1}}/{v_{\ell}})\bigg)=\phi_\eps(h/k)+O(\eps\log k),\\
  \end{align*}
  where
  \begin{align*}
    \phi_\eps(h/k) &  :  = \sum_{\ell=1}^{r}\bigg(\frac{\tn{Vol} (K)}{2\pi}\frac{v_{\ell}}{v_{\ell-1}}+\frac{3}2\log ({v_{\ell}}/{v_{\ell-1}}) +
    \psi_\eps({v_{\ell-1}}/{v_{\ell}})\bigg).
  \end{align*}
  Letting~$T(x) = \{1/x\}$ for~$x\in(0, 1]$, we note that for~$2\leq s \leq r$, we have~$\frac{v_{s-1}}{{v_s}} = T^{r-s}(h'/k)$, whereas $\frac{v_{0}}{{v_1}} = T^{r-1}(h'/k)+\1_{b_1=1}$ (the contribution of $\1_{b_1=1}$ being negligible). We apply~\cite{BettinDrappeau} (Theorem~1.3 with~$\lambda=1$, complemented by Theorem~1.2 with~$\gamma(x) = \{1/x\}$), and obtain that the estimate~\eqref{eq:conj} holds with $\log \J_{4_1,0}(\e[{ h}/k])$ replaced by $\phi_\eps(h/k)$ and
  $$ D_K = \frac{1-\gamma_0 - \log 2}{12/\pi^2} + \frac{3\pi}{\tn{Vol}(K)} + \frac{24}{\pi\tn{Vol}(K)}\int_0^1 \frac{\psi_\eps(1/x)\df x}{1+x}, $$
  and with an error term~$o_\eps(N^2)$. The result then follows by letting $\eps\to0^+$ sufficiently slowly with respect to $N$, and making the change of variables $h/k\to h'/k$ on the left hand side of~\eqref{eq:conj}.
\end{proof}

\begin{proof}[Proof of Corollary~\ref{concor}]
  We prove the result in the case where $x\in[0,1]\setminus\Q$ and for $y\to x^-$, the other cases being analogous. Let $x=[0;b_1,b_2,\dots]$ and let $h/k=[0;b_1,b_2,\dots,b_{2n},X,Y]$ for some $X,Y\in\N_{>0}$. Then, $h/k\to x^{-}$ as $n\to\infty$, uniformly in $X>b_{2n+1}$ and $Y$. We have $k/h\equiv [0;b_2,\dots,b_{2n},X,Y]\mod 1$ and so by~\cite[(1.2)-(1.3) and Lemma~4]{Bettin2015}
  \begin{equation*}
    \frac1h\cz(\overline k/h)=\frac{1}{\pi}\sum_{\ell=1}^{2n+1}\frac{(-1)^\ell\log(v_{\ell-1}/v_{\ell})}{v_{\ell-1}}+O(n)
  \end{equation*}
  uniformly in $X,Y$, where $v_{i}$ denotes the partial quotient of $ [0;b_2,\dots,b_{2n},X,Y]$. Now, let $B=2+\max_{1\leq i\leq 2n}b_i$. Then, for $\ell<2n$, we have 
  $ v_{\ell}/v_{\ell-1}<B$, whereas $v_{2n}/v_{2n-1}=X+O(1)$ and $v_{2n+1}/v_{2n}=Y+O(1)$. Also, $v_{2n}=X{v_{2n-1}}+{v_{2n-2}}$ and $v_{2n-2}\leq v_{2n-1}\leq B^{2n-1}$. It follows, that
  \begin{equation*}
    \frac\pi h\cz(\overline k/h)=\frac{\log Y}{X{v_{2n-1}}+{v_{2n-2}}}-\frac{\log X}{v_{2n-1}}+O(n\log B)=-\frac{\log X}{v_{2n-1}}(1+o(1))
  \end{equation*}
  as $n, X, Y\to\infty$ under the constraint $ Y n B^{2n}=o(\log X)$.
  By~\eqref{eq:srf} and~\cite[Theorem~1]{BettinDrappeaub} we then have
  \begin{equation*}
    H_{4_1}^*(h/k)=-\frac{\log X}{v_{2n-1}}(1+o(1))(1+O(1/Y))+O\Big(Y+\frac{\log X}{Y v_{2n-1}}\Big)
  \end{equation*}
  and this goes to $-\infty$ as $n,X,Y\to\infty$ with $ Y n B^{2n}=o(\log X)$.
\end{proof}

\bibliographystyle{./plain.bst}
\bibliography{../bib.bib}

\begin{thebibliography}{10}

\bibitem{Abel1992}
N.~H. Abel.
\newblock {\em {\OE{}}uvres compl\`{e}tes. {T}ome {I}}.
\newblock \'{E}ditions Jacques Gabay, Sceaux, 1992.
\newblock Edited and with a preface by L. Sylow and S. Lie, Reprint of the
  second (1881) edition.

\bibitem{AndersenHansen2006}
J.~E. Andersen and S.~K. Hansen.
\newblock Asymptotics of the quantum invariants for surgeries on the figure 8
  knot.
\newblock {\em J. Knot Theory Ramifications}, 15(4):479--548, 2006.

\bibitem{Baez-Duarte2003}
L.~B\'{a}ez-Duarte.
\newblock A strengthening of the {N}yman-{B}eurling criterion for the {R}iemann
  hypothesis.
\newblock {\em Atti Accad. Naz. Lincei Rend. Lincei Mat. Appl.}, 14(1):5--11,
  2003.

\bibitem{Baez-DuarteBalazardEtAl2005}
L.~B\'{a}ez-Duarte, M.~Balazard, B.~Landreau, and \'E. Saias.
\newblock \'{E}tude de l'autocorr\'{e}lation multiplicative de la fonction
  `partie fractionnaire'.
\newblock {\em Ramanujan J.}, 9(1-2):215--240, 2005.

\bibitem{Bagchi2006}
B.~Bagchi.
\newblock On {N}yman, {B}eurling and {B}aez-{D}uarte's {H}ilbert space
  reformulation of the {R}iemann hypothesis.
\newblock {\em Proc. Indian Acad. Sci. Math. Sci.}, 116(2):137--146, 2006.

\bibitem{BaladiVallee2005}
V.~Baladi and B.~Vall\'{e}e.
\newblock Euclidean algorithms are {Gaussian}.
\newblock {\em J. Number Theory}, 110(2):331--386, February 2005.

\bibitem{Berndt2010}
B.~C. Berndt.
\newblock What is a {$q$}-series?
\newblock In {\em Ramanujan rediscovered}, volume~14 of {\em Ramanujan Math.
  Soc. Lect. Notes Ser.}, pages 31--51. Ramanujan Math. Soc., Mysore, 2010.

\bibitem{Bettin2015}
S.~Bettin.
\newblock On the distribution of a cotangent sum.
\newblock {\em Int. Math. Res. Not. IMRN}, (21):11419--11432, 2015.

\bibitem{BettinConrey2013}
S.~Bettin and J.~B. Conrey.
\newblock A reciprocity formula for a cotangent sum.
\newblock {\em Int. Math. Res. Not. IMRN}, (24):5709--5726, 2013.

\bibitem{BettinDrappeau}
S.~Bettin and S.~Drappeau.
\newblock Limit laws for rational continued fractions and value distribution of
  quantum modular forms.
\newblock Preprint, https://arxiv.org/abs/1903.00457v1.

\bibitem{BettinDrappeaub}
S.~Bettin and S.~Drappeau.
\newblock Partial sums of the cotangent function.
\newblock Preprint.

\bibitem{CalegariGaroufalidisZagier}
F.~Calegari, S.~Garoufalidis, and D.~Zagier.
\newblock Bloch groups, algebraic {K}-theory, units, and {Nahm's} conjecture.
\newblock preprint.

\bibitem{Champanerkar_etal}
A.~Champanerkar, O.~Dasbach, E.~Kalfagianni, I.~Kofman, W.~Neumann, and
  N.~Stoltzfus, editors.
\newblock {\em Interactions between hyperbolic geometry, quantum topology and
  number theory}, volume 541 of {\em Contemporary Mathematics}. American
  Mathematical Society, Providence, RI, 2011.

\bibitem{DimofteGaroufalidis2018}
T.~Dimofte and S.~Garoufalidis.
\newblock Quantum modularity and complex {C}hern-{S}imons theory.
\newblock {\em Communications in Number Theory and Physics}, 12(1):1--52, 2018.

\bibitem{DimofteGukovEtAl2009}
T.~Dimofte, S.~Gukov, J.~Lenells, and D.~Zagier.
\newblock Exact results for perturbative {C}hern-{S}imons theory with complex
  gauge group.
\newblock {\em Commun. Number Theory Phys.}, 3(2):363--443, 2009.

\bibitem{FlajoletSedgewick2009}
P.~Flajolet and R.~Sedgewick.
\newblock {\em Analytic combinatorics}.
\newblock Cambridge University Press, Cambridge, 2009.

\bibitem{Garoufalidis2018}
S.~Garoufalidis.
\newblock Quantum knot invariants.
\newblock {\em Res. Math. Sci.}, 5(1):Paper No. 11, 17, 2018.

\bibitem{GaroufalidisZagier}
S.~Garoufalidis and D.~Zagier.
\newblock Quantum modularity of the {Kashaev} invariant.
\newblock In preparation.

\bibitem{GZ}
I.~S. Gradshteyn and I.~M. Ryzhik.
\newblock {\em Table of integrals, series, and products}.
\newblock Elsevier/Academic Press, Amsterdam, seventh edition, 2007.
\newblock Translated from the Russian.

\bibitem{Gukov2005}
S.~Gukov.
\newblock Three-dimensional quantum gravity, {C}hern-{S}imons theory, and the
  {A}-polynomial.
\newblock {\em Commun. Math. Phys.}, 255(3):577--627, 2005.

\bibitem{Iwaniec1997}
H.~Iwaniec.
\newblock {\em Topics in classical automorphic forms}, volume~17 of {\em
  Graduate Studies in Mathematics}.
\newblock American Mathematical Society, Providence, RI, 1997.

\bibitem{Jones}
V.~F.~R. Jones.
\newblock Hecke algebra representations of braid groups and link polynomials.
\newblock {\em Ann. of Math. (2)}, 126(2):335--388, 1987.

\bibitem{Kashaev1995}
R.~M. Kashaev.
\newblock A link invariant from quantum dilogarithm.
\newblock {\em Modern Phys. Lett. A}, 10(19):1409--1418, 1995.

\bibitem{Kashaev1997}
R.~M. Kashaev.
\newblock The hyperbolic volume of knots from the quantum dilogarithm.
\newblock {\em Lett. Math. Phys.}, 39(3):269--275, 1997.

\bibitem{KashaevTirkkonen}
R.~M. Kashaev and O.~Tirkkonen.
\newblock Proof of the {Volume} {Conjecture} for {Torus} {Knots}.
\newblock {\em J. Math. Sci.}, 115(1):2033--2036, May 2003.

\bibitem{Kin}
A.~Khintchine.
\newblock Metrische {Kettenbruchprobleme}.
\newblock {\em Compos. Math.}, 1:361--386, 1935.

\bibitem{Murakami2011}
H.~Murakami.
\newblock An introduction to the volume conjecture.
\newblock In {\em Interactions between hyperbolic geometry, quantum topology
  and number theory}, volume 541 of {\em Contemp. Math.}, pages 1--40. Amer.
  Math. Soc., Providence, RI, 2011.

\bibitem{MurakamiMurakami2001}
H.~Murakami and J.~Murakami.
\newblock The colored {J}ones polynomials and the simplicial volume of a knot.
\newblock {\em Acta Math.}, 186(1):85--104, 2001.

\bibitem{MurakamiMurakamiEtAl2002}
H.~Murakami, J.~Murakami, M.~Okamoto, T.~Takata, and Y.~Yokota.
\newblock Kashaev's conjecture and the {C}hern-{S}imons invariants of knots and
  links.
\newblock {\em Experiment. Math.}, 11(3):427--435, 2002.

\bibitem{MurakamiYokota2018}
H.~Murakami and Y.~Yokota.
\newblock {\em Volume {Conjecture} for {Knots}}.
\newblock Springer Singapore, 2018.

\bibitem{Murakami2017}
J.~Murakami.
\newblock Generalized {K}ashaev invariants for knots in three manifolds.
\newblock {\em Quantum Topology}, 8(1):35--73, 2017.

\bibitem{Ohtsuki52}
T.~Ohtsuki.
\newblock On the asymptotic expansion of the {K}ashaev invariant of the {$5_2$}
  knot.
\newblock {\em Quantum Topology}, 7(4):669--735, 2016.

\bibitem{Ohtsuki7}
T.~Ohtsuki.
\newblock On the asymptotic expansions of the {K}ashaev invariant of hyperbolic
  knots with seven crossings.
\newblock {\em Internat. J. Math.}, 28(13):1750096, 143, 2017.

\bibitem{OhtsukiYokota6}
T.~Ohtsuki and Y.~Yokota.
\newblock On the asymptotic expansions of the {K}ashaev invariant of the knots
  with 6 crossings.
\newblock {\em Math. Proc. Cambridge Philos. Soc.}, 165(2):287--339, 2018.

\bibitem{OhtsukiTakata2015}
Tomotada Ohtsuki and Toshie Takata.
\newblock On the {K}ashaev invariant and the twisted {R}eidemeister torsion of
  two-bridge knots.
\newblock {\em Geom. Topol.}, 19(2):853--952, 2015.

\bibitem{Olver1997}
F.~W.~J. Olver.
\newblock {\em Asymptotics and special functions}.
\newblock AKP Classics. A K Peters, Ltd., Wellesley, MA, 1997.
\newblock Reprint of the 1974 original [Academic Press, New York; MR0435697 (55
  \#8655)].

\bibitem{Plana1820}
G.~Plana.
\newblock Sur une nouvelle expression analytique des nombre {Bernoulliens}.
\newblock {\em Memorie della Reale accademia delle scienze di Torino},
  25:403--418, 1820.

\bibitem{Tenenbaum2015a}
G.~Tenenbaum.
\newblock {\em Introduction to analytic and probabilistic number theory},
  volume 163 of {\em Graduate Studies in Mathematics}.
\newblock American Mathematical Society, Providence, RI, third edition, 2015.

\bibitem{Vasyunin1995}
V.~I. Vasyunin.
\newblock On a biorthogonal system associated with the {R}iemann hypothesis.
\newblock {\em Algebra i Analiz}, 7(3):118--135, 1995.

\bibitem{Witten1994}
E.~Witten.
\newblock Quantum field theory and the {J}ones polynomial.
\newblock In {\em Braid group, knot theory and statistical mechanics, {II}},
  volume~17 of {\em Adv. Ser. Math. Phys.}, pages 361--451. World Sci. Publ.,
  River Edge, NJ, 1994.

\bibitem{Yokota2003}
Y.~Yokota.
\newblock From the {J}ones polynomial to the {$A$}-polynomial of hyperbolic
  knots.
\newblock In {\em Proceedings of the {W}inter {W}orkshop of
  {T}opology/{W}orkshop of {T}opology and {C}omputer ({S}endai, 2002/{N}ara,
  2001)}, volume~9, pages 11--21, 2003.

\bibitem{Yokota2011}
Y.~Yokota.
\newblock On the complex volume of hyperbolic knots.
\newblock {\em J. Knot Theory Ramifications}, 20(7):955--976, 2011.

\bibitem{Zagier2010}
D.~Zagier.
\newblock Quantum modular forms.
\newblock In {\em Quanta of maths}, volume~11 of {\em Clay {Math}. {Proc}.},
  pages 659--675. Amer. Math. Soc., Providence, RI, 2010.

\end{thebibliography}

\end{document}